\documentclass[a4paper, 11pt]{article}
\usepackage[utf8]{inputenc}
\usepackage[T1]{fontenc}
\usepackage{tgpagella}
\usepackage{amsmath, amssymb, amsthm, bm, graphicx, mathdots, xcolor}
\usepackage[textwidth = 14cm]{geometry}
\usepackage[pdfa, hidelinks, colorlinks = true, allcolors=blue]{hyperref}
\DeclareMathOperator{\ILP}{ILP}
\DeclareMathOperator{\rank}{rank}
\DeclareMathOperator{\Ima}{Im}
\DeclareMathOperator{\diag}{diag}
\DeclareMathOperator{\slope}{slope}
\DeclareMathOperator{\slopematrix}{Slope}

\newcommand{\R}{\mathbb{R}}
\newcommand{\Z}{\mathbb{Z}}
\newcommand{\N}{\mathbb{N}}

\newcommand{\bA}{\bm{A}}
\newcommand{\bB}{\bm{B}}
\newcommand{\bC}{\bm{C}}
\newcommand{\bD}{\bm{D}}
\newcommand{\bx}{\bm{x}}

\newcommand{\bc}{\bm{c}}
\newcommand{\bb}{\bm{b}}

\newcommand{\be}{\bm{e}}
\newcommand{\bp}{\bm{p}}
\newcommand{\bq}{\bm{q}}
\newcommand{\br}{\bm{r}}
\newcommand{\bs}{\bm{s}}
\newcommand{\bt}{\bm{t}}
\newcommand{\bv}{\bm{v}}
\newcommand{\bw}{\bm{w}}
\newcommand{\bu}{\bm{u}}
\newcommand{\bz}{\bm{z}}
\newcommand{\balpha}{\bm{\alpha}}

\newcommand{\bM}{\bm{M}}
\newcommand{\bN}{\bm{N}}
\newcommand{\bT}{\bm{T}}
\newcommand{\bR}{\bm{R}}
\newcommand{\bU}{\bm{U}}
\newcommand{\bV}{\bm{V}}
\newcommand{\indet}{x}
\newcommand{\intvar}{a}

\newcommand{\polyring}{\Z[\indet]}
\newcommand{\polyS}{S}
\newcommand{\polyM}{\bM}
\newcommand{\id}{\bm{I}}
\newcommand{\zero}{\bm{0}}

\newtheorem{theorem}{Theorem}

\newtheorem{definition}{Definition}
\newtheorem{lemma}{Lemma}
\newtheorem{corollary}{Corollary}

\newtheorem{remark}{Remark}
\newtheorem{example}{Example}

\author{Marcel Celaya\thanks{Cardiff University, School of Mathematics, Wales, United Kingdom, \texttt{celayam@cardiff.ac.uk}} \and Stefan Kuhlmann\thanks{ETH Zürich, Department of Mathematics, Institute for Operations Research, Switzerland, \texttt{\{stefan.kuhlmann,robert.weismantel\}@ifor.math.ethz.ch}} \and Robert Weismantel\footnotemark[2]}

\title{Polynomial Matrices in Integer Programming With Restricted Subdeterminants}

\date{ }

\begin{document}
	\maketitle

	\begin{abstract}
		We introduce a framework for tackling questions in discrete optimization associated with parametric constraint matrices. More precisely, the constraint matrices have entries that are polynomials in one variable and all subdeterminants of these matrices are polynomials in a given prescribed set $S$. 
        Two key problems arise in this context. The first is the recognition problem: can a matrix of this form be recognized in polynomial time? The second is the optimization problem: given an integer program whose constraint matrix is of this form, can it be solved in polynomial time? We answer both questions affirmatively for a particular set $S$ consisting of nine linear forms. The matrices we consider are of themselves independent interest; they arise as matrix projections of certain bimodular matrices that admit two distinct unimodular projections.	
		This paper is an extension of a conference proceedings version by the same authors that was presented at IPCO 2024.
	\end{abstract}


	\medskip

	\section{Introduction}
	
	For a matrix $\bM$, let $\Delta(\bM)$ denote the maximal absolute value of a subdeterminant of $\bM$. It has long been known that the number $\Delta(\bM)$ plays a crucial role in understanding complexity questions related to integer programming problems whose associated constraint matrix is $\bM$. Important examples include bounds on the diameter of polyhedra \cite{bonisummaeisenbranddiameterpoly14,dadushhahnle2016shadowsimplex,narayananshahsri2021spectraldiameter}, questions about the proximity between optimal LP solutions and optimal integral solutions \cite{alievhenkoertel2020knapsackprox,celayakuhlpaatweis2023proxandflatness,CGST1986,lee2020improving}, and bounds on the size of the minimal support of integer vectors in standard form programs \cite{alievAverkovLoeraOertel21,alievdeloera2018supportintegeroptimal,alievdeloesparelindio2017,gribanov2024delta,svp_delta}. 
	Largely unexplored however remain two fundamental algorithmic questions: Given a finite set $S\subseteq \Z$ of allowed values for the subdeterminants of the matrix $\bM$:
	\begin{enumerate} 
		\item What is the complexity of solving an integer linear program with associated constraint matrix $\bM$ in terms of $\left| S\right|$ and $\max_{s\in S} |s|$? (Optimization problem)
		\item How can we efficiently verify whether a matrix $\bM$ has all its subdeterminants in  $S$? (Recognition problem)
	\end{enumerate}
	If $S = \pm\lbrace 0,1\rbrace$, both questions have been answered. Indeed, a famous theorem of Hoffmann and Kruskal \cite{hoffkruskal1956tustatement} and a decomposition theorem of Seymour for totally unimodular matrices \cite{SEYMOUR1980305} allow us to tackle both questions. We refer to \cite{schrijvertheorylinint86} for further theory concerning totally unimodular matrices. If $S=\pm\lbrace 0,1,2\rbrace$, there exists a polynomial time algorithm to solve Question 1 \cite{artweiszenbimodalgo2017,veselovchirkovbimodular09}. In this generality, further optimization results are not known to us. However, more polynomial time algorithms exist under specific matrix restrictions: for instance, if one assumes that each row of the constraint matrix contains at most two non-zero entries \cite{two_nonzeros_per_row,two_non_zeros_most_rows}, or when the constraint matrix is the transpose of a network matrix with a constant number of additional rows \cite{nearly_tu}. A randomized polynomial time algorithm for the related integer feasibility question can be derived if the constraint matrix is $\pm\lbrace 0,\Delta\rbrace$-modular for $\Delta\leq 4$ \cite{naegele2024advancesstrictlydelta,naegelesanzencongruence2023}. There are also polynomial time algorithms for the optimization and recognition problem if the constraint matrix is $\pm\lbrace a,b,c\rbrace$-modular \cite{glanzer2022notes} for $a,b,c \in \N$.  
	
	The main reason why there are only few general results on optimization and recognition problems is due to the lack of understanding of integral matrices with bounded subdeterminants. 
	The motivation of this paper is to develop a framework to investigate the structure of those matrices. Our point of departure is to swap out $\Z$ for the ring $\polyring$ of polynomials with integral coefficients in one indeterminate $\indet$, and consider matrices $\polyM$ over $\polyring$ whose subdeterminants all lie in some fixed set $\polyS\subseteq\polyring$. 
	\begin{definition}
		\label{def_tot_S_mod_forbidden_submatrix}
		Let $R$ be a commutative ring with unity, and let $\polyS\subseteq R$. Let $1\leq m,n$. A matrix $\polyM\in R^{m\times n}$ is \emph{totally $\polyS$-modular} if every $k\times k$ subdeterminant of $\polyM$ is contained in $\polyS$ for $1\leq k\leq \min \lbrace m,n\rbrace$. Given that $m\leq n$, a matrix $\polyM\in R^{m\times n}$ with full row rank is \emph{$\polyS$-modular} if every $m\times m$ subdeterminant of $\polyM$ is contained in $\polyS$.
	\end{definition}

The advantage of operating in $\polyring$ is that we can leverage arithmetic properties of $\polyring$ to prove statements which would not otherwise hold over $\Z$. From this more abstract setting, we can still make progress towards Question 1 and 2 by evaluating the polynomials in $\polyM$ and $\polyS$ at integers. Indeed, if $\polyS\subseteq \polyring$, then we can evaluate the polynomials in $\polyS$ at an integer $\intvar$ to get the set $\polyS(a)\subseteq\Z$, and it is not hard to show that every totally $\polyS$-modular matrix $\polyM$ becomes a totally $\polyS(\intvar)$-modular matrix $\polyM(\intvar)$ after evaluating the entries of $\polyM$ at $\intvar$. Moreover, we show in Theorem~\ref{thm_from_integer_to_polyring} that if $\polyS$ is finite, then the function $\polyM\mapsto\polyM(\intvar)$ is a bijection from totally $\polyS$-modular matrices to totally $\polyS(\intvar)$-modular matrices for all but finitely many integers $\intvar$.

In this paper we primarily focus on the case $S=\pm\{0,1,x,x+1,2x+1\}$, although many of the techniques we develop are more widely applicable. We outline why we think this particular set is interesting in Section~\ref{sec_smallest_nontrivial_cases}. We state our main results in Section~\ref{sec_results}, and provide some background material in Section~\ref{sec_tools}. In Section~\ref{sec_bijection} we identify, for this particular set $\polyS$, those integers $\intvar$ for which the function $\polyM\mapsto\polyM(\intvar)$ described above is not a bijection. In Section~\ref{sec_recognition} we consider the recognition problem for totally $\polyS$-modular matrices. In Section~\ref{sec_optimization} we consider the optimization problem for totally $\polyS(\intvar)$-modular matrices, given some integer $\intvar$. In Section~\ref{sec_bimodular} we show that $\polyS$-modular matrices correspond to a special family of bimodular matrices.
	
	Given $\polyS\subseteq\polyring$, let us remark that $\polyS$-modular matrices are interesting from a matroid theory point of view. For example, the family of matroids admitting a representation as an $\polyS$-modular matrix is closed under taking minors and duality. In \cite{oxleywalsh2022bimodular}, the authors list some excluded minors of the family corresponding to $\polyS=\pm\{0,1,2\}$, and use them to establish a bound on the number of columns that an $\polyS$-modular matrix can have. The more general problem of bounding the number of columns of a $\pm\lbrace 0,1,2,\ldots,\Delta\rbrace$-modular matrix for $\Delta\geq 2$ is an active line of research; see for example \cite{averkovschymura24,GeelenNelsonWalsh2024,paatstallknecht2024forbiddenminors}. In a very different direction, if we take $\polyS=\N$, the matroids we get are known as \emph{positroids}. See \cite{postnikov2006total} for an introduction to this fascinating topic.

\begin{remark}\label{rem_S_modular_total_S_modular_equivalence}
If $1\in\polyS$, then the two notions introduced in Definition~\ref{def_tot_S_mod_forbidden_submatrix} are related: an $m\times n$ matrix $\polyM$ is totally $\polyS$-modular if and only if $(\id\;\polyM')$ is $S$-modular, where $\id$ denotes the $m\times m$ identity matrix and $\bM'_{i,j}=(-1)^{m-i}\bM_{m-i+1,j}$ for each $i\in [m]$ and $j\in [n]$.
\end{remark}

\section{Choosing a Set of Polynomials}
\label{sec_smallest_nontrivial_cases}
Which sets $\polyS\subseteq\polyring$ are interesting to study in the context of totally $\polyS$-modular matrices? We make some stipulations up front: $\polyS$ consists of linear forms only, $\polyS=-\polyS$, and $s\neq 2s'$ for each non-zero pair $s,s'\in \polyS$. In this section, we make the case for $\polyS=\pm\{0,1,x,x+1,2x+1\}$.

The first observation is that we should have $0\in\polyS$, since we would otherwise have only finitely many invertible totally $\polyS$-modular matrices. This follows from, for instance, \cite[Lemma~7]{ARTMANN2016635}. In light of Remark~\ref{rem_S_modular_total_S_modular_equivalence}, or alternatively the standard convention that the determinant of a $0\times 0$ matrix is 1, it also makes sense to have $1\in\polyS$. Finally, we want at least one non-constant polynomial in our set, so we might as well have $x\in\polyS$. So far, this yields $\polyS=\pm\{0,1,x\}$.

	\begin{definition}
	\label{def_forbidden_submatrix}
	Let $2\leq l$. The matrix $\polyM\in \polyring^{l\times l}$ is a \emph{forbidden submatrix for} $\polyS$ if every $(l-1)\times (l-1)$ submatrix is totally $\polyS$-modular but $\det\polyM\notin \polyS$. By $F(\polyS)$, we denote the set of all polynomials that arise as a determinant of some forbidden submatrix for $\polyS$.
	\end{definition}
For the particular case $\polyS=\pm\{0,1,x\}$, $\polyS$-modular (and hence totally $\polyS$-modular) matrices have a very simple structure: they are of the form given in Example~\ref{ex_trivially_S_modular}. To get something more interesting, we can append to $\polyS$ those elements $F(\polyS)$ to $S$ which do not violate the above stipulations. Using Lemma~\ref{lemma_desnanot_jacobi} we calculate
\[
F(\polyS)=\pm\{2,2x,2x^2,x-1,x+1,x^2\}.
\]
We add the elements $x-1$ and $x+1$ to $\polyS$ to get $\polyS'=\pm\{0,1,x-1,x,x+1\}$. We show in Lemma~\ref{lem_equivalence_of_S_and_S1_modular} that totally $\polyS'$-modular matrices and totally $\pm\{0,1,x,x+1,2x+1\}$-modular matrices are essentially equivalent.

Here is another way to arrive at our desired set. Given that $\polyS\subseteq\polyring$ consists only of linear forms, what are the maximal $\polyS$ for which there exist two distinct integers $a<b$ such that $\polyS(a)=\polyS(b)=\{-1,0,1\}$? Such $\polyS$ are interesting because non-constant totally $\polyS$-modular matrices admit two distinct totally unimodular evaluations. Each element of such an $\polyS$ corresponds to a line in the plane which meets exactly one point in $\{(a,-1),(a,0),(a,1)\}$ and exactly one point in $\{(b,-1),(b,0),(b,1)\}$. There are nine such lines in total. Moreover, by integrality we must have $b-a=1$. Replacing $x$ with $x-b$, we can assume without loss of generality that $b=0$ and $a=-1$. Our nine lines are now defined precisely by the linear forms $\pm\{0,1,x,x+1,2x+1\}$. This connection to total unimodularity is the key to the recognition algorithm given in Section~\ref{sec_recognition}.
\section{Results}
	\label{sec_results}
	Let us discuss our main results, and comment on how the results of this paper extend the results of the extended abstract \cite{celayakuhlmannweismantel2024matricespolynomial}. 
	Our first main result is that, for some particular sets $\polyS\subseteq\polyring$, we can
	recognize total $\polyS$-modularity in polynomial time. 

	\begin{theorem}\label{thm_recognition_polyring}
		Let $\bM\in \polyring^{m\times n}$ and $\polyS\subseteq \polyring$ be one of the following sets:
		\begin{enumerate}
			\item $\pm \lbrace 0,1,\indet\rbrace$, 
			\item $\pm \lbrace 0,1,\indet,\indet + 1\rbrace$, and
			\item $\pm \lbrace 0,1,\indet,\indet + 1, 2\indet + 1\rbrace$.
		\end{enumerate} 
		Then one can decide in polynomial time whether $\bM$ is totally $\polyS$-modular.
	\end{theorem}
	Our next result concerns integer linear programs of the following form, where we assume that $\bN$ is a full column rank matrix and $\bb$ and $\bc$ are integral:
	\begin{align*}
		\ILP(\bN,\bb,\bc): \,\max \bc^\top\bx\, \ \text{s.t.}\, \ \bN\bx\leq\bb, \ \bx\in\Z^n.
	\end{align*}
	Given $\intvar\in\Z$, $\polyS\subseteq\polyring$ and a totally $\polyS$-modular matrix $\bM\in \polyS^{m\times n}$, we define
$\polyS(\intvar)\subseteq\Z$ to be the set of evaluations of all polynomials
of $\polyS$ at $\intvar$, and similarly we define $\bM(\intvar)\in(\polyS(\intvar))^{m\times n}$
to be the matrix obtained from $\bM$ by evaluating each entry at $\intvar$.
		
	\begin{theorem}\label{thm_opt_polyring}
		Let $\polyS = \pm \lbrace 0,1,\indet,\indet + 1, 2\indet + 1\rbrace$. Let $\bM\in \polyring^{m\times n}$ be totally $\polyS$-modular. Then one can solve in polynomial time $\ILP(\bM(\intvar),\bb,\bc)$ for integral $\bb$, $\bc$, and all $\intvar\in\Z$ such that $\bM(\intvar)$ has full column rank.
	\end{theorem}
	
		To formulate the next result, recall from Definition~\ref{def_forbidden_submatrix} that $F(\polyS)$ denotes the set of determinants of forbidden submatrices for $\polyS$. Define
	\begin{align*}
		I(\polyS) := \lbrace \intvar\in\Z : s(\intvar) = f(\intvar)\text{ for some }s\in \polyS \text{ and }f\in F(\polyS)\rbrace\subseteq\Z.
	\end{align*}
	In other words, $I(\polyS)$ consists of all $\intvar\in\Z$ such that there exists a forbidden matrix $\bM$ for $\polyS$, for which $\bM(\intvar)$ is totally $\polyS(\intvar)$-modular.
	\begin{theorem}
		\label{thm_from_integer_to_polyring}
		Let $S\subseteq \polyring$ be finite and $0\in S$. Then the function $\polyM\mapsto\polyM(\intvar)$ is a bijection from totally $\polyS$-modular matrices to totally $\polyS(\intvar)$-modular matrices if and only if $\intvar \in \Z\backslash I(\polyS)$. Moreover, $I(\polyS)$ is finite, which means this function is not a bijection for at most finitely many values of $\intvar\in\Z$.
	\end{theorem}
	Given particular choices of $\polyS$, we show in  Section~\ref{ss_computing_I(S)} that it is possible to compute $I(\polyS)$ explicitly. This leads to the following two corollaries:
	\begin{corollary}
		\label{cor_recognition_int}
		Let $\bN\in\Z^{m\times n}$ and $\polyS \subseteq \Z$ be one of the following sets
		\begin{enumerate}
			\item $\pm \lbrace 0,1,\intvar\rbrace$ for all $\intvar\in\Z\backslash\lbrace -2,2\rbrace$,
			\item $\pm \lbrace 0,1,\intvar,\intvar + 1\rbrace$ for all $\intvar \in \Z \backslash\lbrace -3,-2,1,2\rbrace$, and
			\item $\pm \lbrace 0,1,\intvar,\intvar + 1,2\intvar + 1\rbrace$ for all $\intvar \in \Z\backslash\lbrace -3,-2,1,2\rbrace$.
		\end{enumerate}
		Then one can decide in polynomial time whether $\bN$ is totally $\polyS$-modular.
	\end{corollary}
	
	\begin{corollary}
		\label{cor_opt_int}
		Let $\intvar\in\Z\backslash\lbrace -3,-2,1,2\rbrace$ and $S = \pm\lbrace 0,1,\intvar,\intvar + 1,2\intvar + 1\rbrace$. Let $\bN\in\Z^{m\times n}$ have full column rank and be totally $S$-modular. Then one can solve $\ILP(\bN,\bb,\bc)$ for integral $\bb$ and $\bc$ in polynomial time.
	\end{corollary}

	Now let $\polyS = \pm \lbrace 0,1,\indet,\indet + 1, 2\indet + 1\rbrace$. The last result of this paper shows that $\polyS$-modular matrices correspond to bimodular matrices which admit two distinct unimodular matrix projections. We define matrix projections in Section~\ref{sec_matrix_projections}.
	\begin{theorem}\label{thm_bimodular_matrix_projection}Assume that $\bA$ is a full row rank matrix with entries in $\polyring$ with at least two distinct non-zero maximal subdeterminants in absolute value. Then $\bA$ is $S$-modular if and only if there exists a full row rank bimodular matrix $\bT$ and non-zero, non-parallel integer vectors $\bp,\bq$ such that $\bT/\bp$ and $\bT/\bq$ are unimodular and $\bA=\bT/(x\bp+(x+1)\bq)$.
\end{theorem}

We now comment on how these results relate to and extend the ones in \cite{celayakuhlmannweismantel2024matricespolynomial}. Theorem~\ref{thm_recognition_polyring} and Corollary~\ref{cor_recognition_int} are new, although Theorem~\ref{thm_recognition_polyring}(3) appears as \cite[Lemma~8]{celayakuhlmannweismantel2024matricespolynomial} and Corollary~\ref{cor_recognition_int}(3) appears as \cite[Theorem~2]{celayakuhlmannweismantel2024matricespolynomial}. Theorem~\ref{thm_opt_polyring} puts \cite[Theorem~3]{celayakuhlmannweismantel2024matricespolynomial} in a more general context, and, in addition, both Theorem~\ref{thm_opt_polyring} and Corollary~\ref{cor_opt_int} generalize \cite[Theorem~3]{celayakuhlmannweismantel2024matricespolynomial} by adding $\pm 1$ to the set $\polyS$. Theorem~\ref{thm_from_integer_to_polyring} is an extension of \cite[Lemma~3]{celayakuhlmannweismantel2024matricespolynomial}. Finally, Theorem~\ref{thm_bimodular_matrix_projection} is new.
	
	\section{Tools}\label{sec_tools}
	Throughout this paper we work with the lexicographical order of $\polyring$ which is defined by $s < t$ if and only if the leading coefficient of $t - s$ is positive. With respect to this ordering we further define the absolute value of $s\in \polyring$ to be
	\begin{align*}
		|s| = \begin{cases}
			s, & \text{if }s \geq 0,\\
			-s, & \text{otherwise} 
		\end{cases}.
	\end{align*}
	Also, we interchangeably pass from polynomials $s\in \polyring$ to polynomial functions, denoted by $s(a)$, when evaluating polynomials. This can be done since the polynomial function is uniquely determined by the polynomial over $\polyring$; see \cite[Chapter~4]{langalgebra}. 
	Let $I\subseteq \lbrack n\rbrack$ and $J \subseteq \lbrack n\rbrack$. We denote by $\bM_{\backslash I,\backslash J}$ the submatrix of $\bM$ without the rows indexed by $I$ and columns indexed by $J$. In what follows, the notation $\bA_{(k)}\subseteq\bM$ is shorthand for specifying that $\bA_{(k)}$ is a $k\times k$ submatrix of $\bM$. 
	Given a submatrix $\bA$ of $\bM$ and $i,j\in\lbrack n\rbrack$, we denote by $\bA[i,j]$ the submatrix of $\bM$ that contains $\bA$ along with the extra row and column indexed by $i$ and $j$ respectively. We write $\id$ for the identity matrix with respect to the appropriate dimension. Most of the results presented below remain valid for rings more general than $\polyring$. However, to remain close to our applications in integer programming, we state them throughout in the setting of $\polyring$ and only deviate if necessary.
	
	\subsection{Determinants of Forbidden Submatrices}
	One of the main tools of this paper is the Desnanot-Jacobi identity, also known as Dodgson condensation, which is a special case of Sylvester's determinant identity \cite{sylvester1851relationminors}. By convention, the determinant of an empty matrix is 1.
	\begin{lemma}[Desnanot-Jacobi identity]
		\label{lemma_desnanot_jacobi}
		Let $\bM\in \polyring^{n\times n}$ for $n\geq 2$ and let $\bA_{(n-2)}\subseteq \bM$ with $\bA_{(n-2)} = \bM_{\backslash I,\backslash J}$ for the ordered sets $I=\lbrace i_1,i_2\rbrace$ and $J=\lbrace j_1,j_2\rbrace$. Then we get
		\begin{align*}
		\det\bM\cdot\det\bA_{(n-2)} = \det\begin{pmatrix}
			\det \bA_{(k)}[i_1,j_1]& \det \bA_{(k)}[i_1,j_2] \\
			\det \bA_{(k)}\lbrack i_2, j_1\rbrack & \det \bA_{(k)}\lbrack i_2,j_2\rbrack
		\end{pmatrix}.
	\end{align*}
	\end{lemma}

	As an illustration of the utility of this identity, we establish a generalization of the well-known fact, usually attributed to Gomory \cite{CORNUEJOLS200097}, that every matrix that is not totally unimodular and has entries in $\pm\lbrace 0,1\rbrace$ contains a submatrix with determinant two in absolute value. 
	
	\begin{lemma}{\cite[Lemma~2]{celayakuhlmannweismantel2024matricespolynomial}}
		\label{lemma_finite_number_of_determinants}
		Let $\polyS\subseteq \polyring$ be finite. Then the set $F(\polyS)$ of determinants attained by the forbidden submatrices of $\polyS$ is finite and for each $f\in F(\polyS)$ we have
		\begin{align*}
			|f|\leq 2\cdot \max_{s\in \polyS}s^2.
		\end{align*} 
	\end{lemma}
	\begin{proof}
		Select some forbidden submatrix $\polyM$ for $\polyS$ such that $\det\polyM\neq 0$. 
		There exists an invertible submatrix $\bA_{(n-2)}\subseteq \bM$. By the Desnanot-Jacobi identity, Lemma~\ref{lemma_desnanot_jacobi}, applied to $\bA_{(n-2)}$, we obtain that 
		\begin{align*}
			\det\polyM = \frac{1}{\det\bA_{(n-2)}}\cdot\left( s_1s_2-s_3s_4\right),
		\end{align*}
		where $s_i\in \polyS$ for $i=1,2,3,4$. Since $\det\bA_{(n-2)}\in S$, the right hand side only attains finitely many values as $\polyS$ is finite. Hence, there are only finitely many values possible for $\det\polyM$. 
		To obtain the inequality, we take absolute values on both sides, apply the triangle inequality, and observe that $\left|s_i\right|\leq \max_{s\in \polyS}\left|s\right|$. The claim follows then from $1\leq \left|\det\bA_{(n-2)}\right|$. 
	\end{proof}
	
	We immediately obtain that every forbidden submatrix for totally $\pm\lbrace 0, 1,\ldots,\Delta\rbrace$-modular matrices over $\Z$ admits a determinant bounded by $2\Delta^2$ in absolute value. The bound in Lemma \ref{lemma_finite_number_of_determinants} is tight: Let $\polyS\subseteq\polyring$ be finite such that $s\in \polyS$ implies $-s\in \polyS$ and $\lbrace 0\rbrace \neq \polyS$. Select $\tau = \max_{s\in S}\left|s\right|$. Then
	\begin{align*}
		\begin{pmatrix} \phantom{-}\tau & \tau \\ -\tau & \tau\end{pmatrix}
	\end{align*}
	is a forbidden submatrix for $S$ and has determinant $2\tau^2$. However, it is possible to strengthen the bound in Lemma \ref{lemma_finite_number_of_determinants} if the dimension is sufficiently large and $S\subseteq \Z$, see \cite[Theorem 4]{celayakuhlmannweismantel2024matricespolynomial}.

We end this subsection with one more determinant identity. Given some $\polyS$, the following analogue of the Desnanot-Jacobi identity is more useful when working with $\polyS$-modular matrices rather than totally $\polyS$-modular matrices:

\begin{lemma}[{3-term Grassmann-Pl\"ucker relations \cite[Ch. 3]{BLSWZ1999}}] \label{lemma_3_term_GP_relations}
Let $m\geq 2$, let $\bA$ be an $m\times (m-2)$ matrix, and let $\bu,\bv,\br,\bs$ be $m$-dimensional vectors. Then 
\[
\det(\bu\ \bv\ \bA)\det(\br\ \bs\ \bA)=\det(\bu\ \br\ \bA)\det(\bv\ \bs\ \bA)-\det(\bu\ \bs\ \bA)\det(\bv\ \br\ \bA).
\]
\end{lemma}
Note that we can derive the Desnanot-Jacobi identity, Lemma~\ref{lemma_desnanot_jacobi}, from this lemma by taking $\bu$ and $\bv$ to be standard unit vectors.
	\subsection{Determinants and Rank-1 Updates}
	Suppose we have a totally $\polyS$-modular matrix $\polyM = \bM(0) + \indet\cdot \bR\in\polyring^{m\times n}$. The following result relates the rank of $\bR$ to the largest degree among the polynomials in a given set $\polyS$. We write $\deg(s)$ for the degree of $s\in \polyring$.
	
	\begin{lemma}{\cite[Lemma~4]{celayakuhlmannweismantel2024matricespolynomial}}
		\label{lemma_rank_coefficient_matrix}
		Let $S\subseteq \polyring$ be finite. Let $\polyM = \bM(0) + \indet\cdot \bR \in\polyring^{m\times n}$ be totally $\polyS$-modular. Then we have $\rank\bR\leq \max_{s\in \polyS}\deg(s)$. 
	\end{lemma}
	\begin{proof}
		Let $r = \rank \bR$ and $\bR'\in\Z^{r\times r}$ be an invertible submatrix of $\bR$ and $\bM' = \bM(0)' + \indet \cdot \bR'$ the corresponding submatrix of $\bM$.  
		We get
		\begin{align*}
			\det\bM' = \det\left(\bM(0)' + \indet \cdot \bR'\right) = \det\bR'\cdot\det\left(\bR'^{-1}\bM(0)' + \indet \cdot \id\right),
		\end{align*} 
		Using the Leibniz formula for the determinant $\det(\bR'^{-1}\bM(0)' + \indet \cdot \bm{I})$, we observe that there exists only one term with degree $r$, namely the term corresponding to the identity permutation, and no term with degree larger than $r$. Thus, the right hand side in the equation above is a polynomial of degree $r$. So $\det\polyM'$ is also a polynomial of degree $r$. Since the matrix $\polyM'$ is totally $\polyS$-modular, we have that $r = \deg(\det\polyM')\leq \max_{s\in \polyS}\deg(s)$. 
	\end{proof}
	
	So, if $\max_{s\in \polyS}\deg(s) = 1$ and $\bR$ is non-zero, then Lemma~\ref{lemma_rank_coefficient_matrix} implies that $\rank\bR = 1$. Hence, we may write $\bM = \bM(0) + \indet\cdot \bu\cdot \bv^\top$ for some integral $\bu$ and $\bv$. 
	Given such a matrix $\bM = \bM(0) + \indet\cdot \bu\cdot \bv^\top$ without any restriction on its subdeterminants, it is possible to establish that subdeterminants of $\polyM$ are indeed polynomials of degree at most one. This is a consequence of the well-known matrix determinant lemma for rank-1 updates. A variant of this lemma is stated below, where we work over $\R[\indet]$ instead of $\polyring$ to deal with the case when $\bM(0)$ is singular. It has been presented in \cite[Lemma~5]{celayakuhlmannweismantel2024matricespolynomial} without a proof.
	
	\begin{lemma}[Modified matrix determinant lemma]
		\label{lemma_matrix_determinant_lemma}
		Let $\polyM = \bM(0)+ \indet \cdot\bu\cdot\bv^\top\in\R[\indet]^{n\times n}$ for $\bM(0)$, $\bu$, and $\bv$. Then we have
		\begin{align*}
			\det \polyM = (\det\bM(0) - \det\bM(-1))\cdot \indet + \det\bM(0).
		\end{align*}
	\end{lemma}
	\begin{proof}
		If $\bM(0)$ is invertible, the classical matrix determinant lemma yields 
		\begin{align*}
			\det\left(\bM(0)+\indet\cdot\bu\cdot\bv^\top\right) = \left(1+\indet\bv^\top \bM(0)^{-1}\bu\right)\det\bM(0).
		\end{align*}
		In particular, evaluating at $\indet = -1$ gives
		\[
		\det\bM(-1)=\left(1-\bv^\top \bM(0)^{-1}\bu\right)\det\bM(0).
		\]
		Hence
		\[
		\det\bM(0)-\det\bM(-1)
		=\det\bM(0)\bv^\top \bM(0)^{-1}\bu,
		\]
		and the claim follows. If $\bM(0)$ is not invertible, consider the matrices
		\[
		\bM_\varepsilon := \left(\bM(0) + \varepsilon\cdot \id\right) + \indet \cdot\bu\cdot\bv^\top\in\R[\indet]^{n\times n},
		\]
		where $\varepsilon\in\mathbb{R}$.
		For all sufficiently small $\varepsilon\neq 0$, the matrix $\bM(0) + \varepsilon\cdot \id$ is invertible. Applying the already proven case, we obtain
		\[
		\det\bM_\varepsilon
		=
		\det(\bM(0) + \varepsilon\cdot \id)
		+ \indet\left(\det(\bM(0) + \varepsilon\cdot \id)-\det(\bM(0) + \varepsilon \cdot\id-\bu\bv^\top)\right).
		\]
		Fix some $a\in\R$ and consider the two sides above as polynomials in $\indet$ evaluated at $a$. Letting $\varepsilon$ go to $0$ shows that 
		\begin{align*}
			\det\bM(a) = \det\bM(0) + a\cdot\left(\det\bM(0) - \det\bM(-1)\right)
		\end{align*}
		for all $a\in \R$, which implies the claimed identity over $\R[x]$.
	\end{proof}
	
	\subsection{Variable Substitution}
	
	For a commutative unital ring $R$, let $\mathcal{M}(R)$ be the set of matrices with entries in $R$; see \cite[Chapter 2]{langalgebra} for basics about rings. Given commutative unital rings $R,R'$ and a ring homomorphism $\varphi:R\rightarrow R'$, we define $\Phi:\mathcal{M}(R)\rightarrow \mathcal{M}(R')$ to be the function that applies $\varphi$ entrywise to matrices in $\mathcal{M}(R)$. In this paper we are primarily concerned with the case $R=\polyring$, but it will also be convenient to consider the more general setting $R=\mathbb{Q}(x)$ where $\mathbb{Q}(x)$ is the field of fractions of $\polyring$. Note that we have the inclusions $\Z\hookrightarrow\polyring\hookrightarrow\mathbb{Q}(x)$. 
	
	It is a basic fact that every homomorphism $\mathbb{Q}(x)\rightarrow\mathbb{Q}(x)$ is uniquely determined by the value it assigns to $x$. However, not every $r\in\mathbb{Q}(x)$ has the property that the assignment $x\mapsto r$ extends to a homomorphism $\mathbb{Q}(x)\rightarrow \mathbb{Q}(x)$. For example, there is no homomorphism $\varphi:\mathbb{Q}(x)\rightarrow\mathbb{Q}(x)$ for which $\varphi(x)=2$, because we would have
\begin{align*}
0=\varphi(x-2)\varphi\left(\frac{1}{x-2}\right)=\varphi\left(\frac{x-2}{x-2}\right)=\varphi(1)=1.
\end{align*}
It is precisely the constants that are problematic here: one can check that $x\mapsto r$ extends to a homomorphism if and only if $r\in\mathbb{Q}(x)\setminus \mathbb{Q}$. 

Nevertheless, in this paper we would like to be able to substitute $x$ with constants. We get around this problem by using localization: given $r\in\mathbb{Q}(x)$ we define $\mathsf{S}_r:=\{q\in\polyring : q(r)\neq 0\}$ and consider the ring $\mathsf{S}_r^{-1}\polyring$ of rational functions of the form $p/q$ with $p\in\polyring$ and $q\in\mathsf{S}_r$. Note that we have the inclusions $\polyring\hookrightarrow\mathsf{S}_r^{-1}\polyring\hookrightarrow\mathbb{Q}(x)$, and note also that $\mathsf{S}_r^{-1}\polyring=\mathbb{Q}(x)$ if and only if $r\in\mathbb{Q}(x)\setminus\mathbb{Q}$. 
	
	Now, for \emph{every} $r\in\mathbb{Q}(x)$, the assignment $x\mapsto r$ extends uniquely to a ring homomorphism $\varphi_r:\mathsf{S}_r^{-1}\polyring\rightarrow\mathbb{Q}(x)$. We define $\Phi_r$ to be the function which applies $\varphi_r$ entrywise to matrices in $\mathcal{M}(\mathsf{S}_r^{-1}\polyring)$. Given such a matrix $\polyM$, we write $\polyM(r)$ as shorthand for $\Phi_r(\polyM)$. Similarly, given a subset $\polyS\subseteq\mathsf{S}_r^{-1}\polyring$, we write $\polyS(r)$ as shorthand for $\varphi_r(\polyS)$. In summary, $\polyM(r)$ and $\polyS(r)$ are what one gets after plugging in $r$ to $\polyM$ and $\polyS$, respectively. We need to assume that the entries of $\polyM$ and the elements of $\polyS$ allow for the substitution of $x$ with $r$, but this is always the case if $r$ is not a constant or if $\polyM\in\mathcal{M}(\polyring)$ and $\polyS\subseteq\polyring$ .

	Fix some $r\in\mathbb{Q}(\indet)$. We now state some straightforward consequences of the fact that $\varphi_r$ is a homomorphism. Lemma~\ref{lemma_determinant_homomorphism} below follows from the Leibniz formula for the determinant, which exhibits the determinant of a square matrix as a polynomial function of its entries. Lemma~\ref{lemma_totally_S_r_modular} is an immediate consequence of Lemma~\ref{lemma_determinant_homomorphism}. Lemma~\ref{lemma_variable_substitution_inverse} follows from the adjugate formula for the matrix inverse, and from Lemma~\ref{lemma_variable_substitution_multiplication} by setting $\bA=\bM$ and $\bB=\bM^{-1}$.
\begin{lemma} For a square matrix $\bM\in\mathcal{M}(\mathsf{S}_r^{-1}\polyring)$ we have $\det (\polyM(r)) = (\det \polyM)(r)$.
\label{lemma_determinant_homomorphism}

\end{lemma}

\begin{lemma}
\label{lemma_totally_S_r_modular} Let $\bM\in\mathcal{M}(\mathsf{S}_r^{-1}\polyring)$ and $\polyS\subseteq\mathsf{S}_r^{-1}\polyring$. If $\bM$ is $\polyS$-modular then $\bM(r)$ is $\polyS(r)$-modular. If $\bM$ is totally $\polyS$-modular then $\bM(r)$ is totally $\polyS(r)$-modular. 
\end{lemma}
		
\begin{lemma}
\label{lemma_variable_substitution_multiplication}Suppose $\bA,\bB\in\mathcal{M}(\mathsf{S}_r^{-1}\polyring)$ are two matrices for which the multiplication $\bA\bB$ exists. Then $\bA\bB\in\mathcal{M}(\mathsf{S}_r^{-1}\polyring)$ and $(\bA\bB)(r)=\bA(r)\bB(r)$.
\end{lemma}

\begin{lemma}
\label{lemma_variable_substitution_inverse}Suppose $\bM\in\mathcal{M}(\mathsf{S}_r^{-1}\polyring)$ is an invertible matrix. Then $\bM^{-1}\in\mathcal{M}(\mathsf{S}_r^{-1}\polyring)$ if and only if $\bM(r)$ is invertible. If this is the case we have $\bM^{-1}(r)=(\bM(r))^{-1}$.
\end{lemma}
We use in Section~\ref{sec_bimodular} special cases of the following example of variable substitution.
\begin{example}
Suppose  $a,b,c,d$ are integers such that $ad-bc\neq 0$. We can define a map $\mathbb{Q}(x)\rightarrow\mathbb{Q}(x)$ by extending the assignment
\begin{align*}
x\mapsto\frac{ax+b}{cx+d}
\end{align*}
to all of $\mathbb{Q}(x)$. This map is invertible, and its inverse is obtained by extending
\begin{align*}
x\mapsto-\frac{dx-b}{cx-a}
\end{align*}
to all of $\mathbb{Q}(x)$. We remark that functions of the form
\begin{align*}
f(x)=\frac{ax+b}{cx+d}
\end{align*}
are called linear fractional transformations, or M\"obius transformations if $a,b,c,d$ are complex numbers and $x$ is a complex variable.
\end{example}
Now suppose $\polyS=\pm\{0,1,x,x+1,2x+1\}$ and $\polyS'=\pm\{0,1,x-1,x,x+1\}$. As an application of the above discussion, we derive a relationship between totally $S$-modular and totally $S'$-modular matrices. 

\begin{lemma}
\label{lem_equivalence_of_S_and_S1_modular}
Fix integers $m,n\geq 1$, and let $I=J=[k]$ for some positive integer $k\leq \min \lbrace m,n\rbrace$. Then the function
\begin{align*}
\bM=\begin{bmatrix}\bA & \bB\\
\bC & \bD
\end{bmatrix}\mapsto\begin{bmatrix}\bA^{-1} & \bA^{-1}\bB\\
-\bC\bA^{-1} & \bD-\bC\bA^{-1}\bB
\end{bmatrix}\left(\frac{x}{x+1}\right),\qquad \bA=\bM_{I,J}
\end{align*}
is a bijection from totally $\polyS'$-modular $m\times n$ matrices $\bM$ such that $\det \bM_{I,J}=\pm(x-1)$, to totally $\polyS$-modular $m\times n$ matrices $\bM'$ such that $\det \bM'_{I,J}=\pm(x+1)$.
\end{lemma}
\begin{proof}
We prove the analogous statement for $S'$-modular and $S$-modular matrices: given positive integers $m<n$ and distinct subsets $B,B'\subseteq [n]$ of size $m$, the function $\bA\mapsto\hat{\bA}:=(\bA_{B'}^{-1}\bA)(\tfrac{x}{x+1})$ defines, up to determinant $1$ row operations, a bijection from $S$-modular $m\times n$ matrices $\bA$ such that $\det \bA_{B}=1$ and $\det \bA_{B'}=\pm(x-1)$, to $S'$-modular $m \times n$ matrices $\hat{\bA}$ such that $\det \hat{\bA}_{B}=\pm(x+1)$ and $\det \hat{\bA}_{B'}=1$. To avoid confusion, we emphasize that we are repurposing the letters $m$, $n$, and $\bA$ from what is given in the lemma statement. We can return to the totally $S'$-modular setting by removing the $B$-columns from $\bA^{-1}_B\bA$, and similarly we can return to the totally $S$-modular setting by removing the $B'$-columns from $\hat{\bA}^{-1}_{B'}\hat{\bA}$.

Let $\bA$ be as in the preceding paragraph, and let $\hat{\bA}=(\bA_{B'}^{-1}\bA)(\tfrac{x}{x+1})$. By Lemmas~\ref{lemma_variable_substitution_multiplication} and \ref{lemma_variable_substitution_inverse}, we have $\hat{\bA}=((\bA(\tfrac{x}{x+1}))_{B'})^{-1}\bA(\tfrac{x}{x+1})$. By Lemma~\ref{lemma_determinant_homomorphism}, we have $\det(\bA(\tfrac{x}{x+1}))_{B'}=\pm\tfrac{1}{x+1}$ and by Lemma~\ref{lemma_totally_S_r_modular}, we have $\bA(\tfrac{x}{x+1})$ is a $\pm\{0,1,\tfrac{1}{x+1},\tfrac{x}{x+1},\tfrac{2x+1}{x+1}\}$-modular matrix. So we conclude that $\det\hat{\bA}_B=\pm(x+1)$ and $\det\hat{\bA}_{B'}=1$ and that $\hat{\bA}$ is $S$-modular. 

Conversely, suppose $\hat{\bA}$ is an $S$-modular matrix where $\det\hat{\bA}_B=\pm(x+1)$ and $\det\hat{\bA}_{B'}=1$. We define $\bA:=(\hat{\bA}_{B}^{-1}\hat{\bA})(-\tfrac{x}{x-1})$. A similar argument to above shows that $\bA$ is $S'$-modular, and that $\det \bA_B=1$ and $\det \bA_{B'}=\pm(x-1)$. It is moreover straightforward to check that the maps $\bA\mapsto\hat{\bA}$ and $\hat{\bA}\mapsto\bA$ are inverses of eachother.
\end{proof}

\subsection{Matrix Projections}\label{sec_matrix_projections}

We now define an operation on full row rank matrices which is the matrix analogue of the \emph{elementary quotient} from matroid theory~\cite[p. 272]{oxley2011matroid}.

\begin{definition}
Let $m,n\geq 2$, let $\bT\in(\mathbb{Q}(x))^{m\times n}$ have full row rank, and let $\bw\in (\mathbb{Q}(x))^{m}$ be non-zero. Define $\bT/\bw$ to be any $(m-1)\times n$
matrix with entries in $\mathbb{Q}(x)$ for which
\begin{align*}
\det((\bT/\bw)_{I}) & =\det(\bw\ \bT_{I})
\end{align*}
for all $I\in\binom{\left[n\right]}{m-1}$. If $\bT$ has full column rank and $\bw^\top$ is a non-zero row vector, then define $\bT/\bw^\top$ to be $(\bT^\top/\bw)^\top$. We call $\bT/\bw$ a \emph{matrix projection}.
\end{definition} 

Such a matrix always exists: choose an $m\times m$ matrix $\bU$ with determinant $1$ such that $\bU\bw=\be^1$, where $\be^1$ is the first standard unit vector, then take the bottom $(m-1)$ rows of $\bU\bA$. This definition does not specify $\bT/\bw$ uniquely, but if $\bA$ and $\bA'$ are two representatives of $\bT/\bw$, then there exists a determinant $1$ matrix $\bV$ such that $\bA=\bV\bA'$. We write $\bA=\bT/\bw$ to say that $\bA$ is a representative of $\bT/\bw$. 
If the entries of $\bT$ and $\bw$ are taken from $\polyring$ and $\gcd(\bw)=1$, then we can take the matrix $\bU$ above to be a matrix over $\polyring$, hence we can take $\bT/\bw$ to be a matrix over $\polyring$.
    
\begin{example}
Suppose
\[
\bT=\left(\begin{array}{cccc}
\phantom{-}1 & \phantom{-}1 & \phantom{-}0 & \phantom{-}1\\
-1 & \phantom{-}0 & \phantom{-}1 & -1\\
\phantom{-}1 & \phantom{-}1 & -1 & \phantom{-}0
\end{array}\right),\qquad\bw=\left(\begin{array}{c}
1\\
1\\
x
\end{array}\right).
\]
Let $\bU\in(\polyring)^{3\times 3}$ be the unimodular matrix for which
\[
\bU\left(\bw\ \bT\right)=\left(\begin{array}{ccccc}
\phantom{-}1 & \phantom{-}1 & \phantom{-}1 & \phantom{-}0 & \phantom{-}1\\
\phantom{-}0 & -2 & -1 & \phantom{-}1 & -2\\
\phantom{-}0 & 1-x & 1-x & -1 & -x
\end{array}\right).
\]
We can take $\bT/\bw$ to be the bottom-right $2\times 4$ submatrix:
\[
\bT/\bw=\left(\begin{array}{cccc}
-2 & -1 & \phantom{-}1 & -2\\
1-x & 1-x & -1 & -x
\end{array}\right).
\]
\end{example}

\begin{example}\label{ex_trivially_S_modular}
Let $\polyS=\pm\{0,s_1,s_2,\ldots,s_k\}\subseteq\polyring$. Let $\bT_1,\bT_2,\ldots,\bT_k$ be full row rank unimodular matrices over $\Z$. Define
\[
\bT:=\left(\begin{array}{cccc}
\bT_1 &  &  &\\
 & \bT_2 & &\\
 & & \ddots &\\
 &  & & \bT_k
\end{array}\right),
\]
For each $i=1,2,\ldots,k$, choose a non-zero column $\bt_i$ from the columns of $\bT$ allocated to the submatrix $\bT_i$, and let $\bw=s_1\bt_1+s_2\bt_2+\cdots+s_k\bt_k$. One can show, using multilinearity of the determinant, that $\bT/\bw$ is $S$-modular, and moreover every non-zero element of $\polyS$ is attained up to sign by a maximal subdeterminant of $\bT/\bw$. 

Matrix projections of this form are considered in \cite{glanzer2022notes}. Specifically, \cite[Theorem~1]{glanzer2022notes} states that if $a\geq b>0$ are integers such that $a\neq2b$, then every $\pm\{0,a,b\}$-modular matrix arises up to column permutations by the above construction.
\end{example}
Matrix projections are compatible with variable substitution: 
\begin{lemma}
\label{lemma_matrix_projection_variable_substitution}
Let $r\in\mathbb{Q}(x)$, and let $\bT\in(\mathsf{S}_r^{-1}\polyring)^{m\times n}$ and $\bw\in(\mathsf{S}_r^{-1}\polyring)^m$ be such that $\bT(r)$ has full row rank and $\bw(r)$ is non-zero. Then any representative $\bA$ of $\bT/\bw$ with entries in $\mathsf{S}_r^{-1}\polyring$ satisfies $\bA(r)=\bT(r)/\bw(r)$.
\end{lemma}
\begin{proof}
Let $I\in\binom{[n]}{m-1}$. Then we have
\[
\det((\bA(r))_I)=(\det(\bA_I))(r)=(\det(\bw\;\bT_I))(r)=\det(\bw(r)\;(\bT(r))_I)
\]
where the first and last equalities hold by Lemma~\ref{lemma_determinant_homomorphism}.
\end{proof}

The following lemma is used in the proofs of Theorem~\ref{thm_opt_polyring} and Theorem~\ref{thm_bimodular_matrix_projection}. For a full rank integer matrix $\bA$, let $\Delta_r(\bA)$ denote the maximum absolute value of any $r\times r$ submatrix of $\bA$ where $r=\rank \bA$.

\begin{lemma}
\label{lemma_bound_subdeterminants_matrix_projection}
Let $m\geq 2$, let $\bT$ be a full row rank integer matrix with $m$ rows, and let $\bp,\bq\in\Z^m$ be non-zero and non-parallel. Then
\[
\Delta_m(\bp\;\bq\;\bT)\leq 2\cdot\Delta_{m-1}(\bT/\bp)\cdot\Delta_{m-1}(\bT/\bq).
\]

\end{lemma}
\begin{proof}
It suffices to only consider invertible $m\times m$ submatrices of $(\bp\;\bq\;\bT)$ which contain either both $\bp$ and $\bq$ or neither $\bp$ nor $\bq$. Suppose $I$ indexes $m-2$ columns of $\bT$ so that $(\bp\;\bq\;\bT_I)$ is invertible. Because $\bT$ has full row rank, there exist columns $\br,\bs$ of $\bT$ such that $(\br\;\bs\;\bT_I)$ is invertible. By the 3-term Grassmann-Pl\"ucker relations (Lemma~\ref{lemma_3_term_GP_relations}), we have
\[
\det(\bp\ \bq\ \bT_{I})\det(\br\ \bs\ \bT_{I})=\det(\bp\ \br\ \bT_{I})\det(\bq\ \bs\ \bT_{I})-\det(\bp\ \bs\ \bT_{I})\det(\bq\ \br\ \bT_{I}).
\]
In particular, we see that $\left|\det(\bp\ \bq\ \bT_{I})\right|$ satisfies the bound since $\det(\br\ \bs\ \bT_{I})$ is a non-zero integer.

Now suppose $B$ indexes $m$ columns of $\bT$ so that $\bT_B$ is invertible. Since $\bp,\bq$ are non-zero and non-parallel, there exists $I\subseteq B$ of size $m-2$ such that $(\bp\;\bq\;\bT_I)$ is invertible. Let $\br,\bs$ be the columns of $\bT$ indexed by $B\backslash I$. Again by the 3-term Grassmann-Pl\"ucker relations, we get $\left|\det(\bT_{B})\right|=\left|\det(\br\ \bs\ \bT_{I})\right|$ satisfies the bound as $\det(\bp\ \bq\ \bT_{I})$ is a non-zero integer.
\end{proof}

	\section{From $\Z$ to $\polyring$}
	\label{sec_bijection}
	Let $\polyS\subseteq \polyring$ be finite and let $\intvar\in\Z$. If one is interested in stuying totally $\polyS(\intvar)$-modular matrices, then a starting point might be to study totally $\polyS$-modular matrices since every such matrix $\bM$ has the property that $\bM(\intvar)$ is totally $\polyS(\intvar)$-modular by Lemma~\ref{lemma_totally_S_r_modular}. However, there are two problems with this approach. First, not every totally $\polyS(\intvar)$-modular matrix is of the form $\bM(\intvar)$ for some totally $\polyS$-modular matrix $\bM$. Second, for a given totally $\polyS(\intvar)$-modular matrix $\bN$, it is possible that there are two distinct totally $\polyS$-modular matrices $\bM,\bM'$ such that $\bM(\intvar)=\bM'(\intvar)=\bN$. Put another way, it is not always the case that, for a given $\intvar\in\Z$, the totally $\polyS$-modular matrices and the totally $\polyS(\intvar)$-modular matrices are in bijection via the map $\bM\mapsto\bM(\intvar)$. The goal of this section is to characterize those $\intvar\in\Z$ for which this failure occurs.

	\begin{example}
		Consider the set $\polyS = \pm\lbrace 0,\indet,\indet + 1,2\indet + 1\rbrace$ and the set $S(1) = \pm \lbrace 0,1,2,3\rbrace$. Take the matrix 
		\begin{align*}
			\polyM(1) = \begin{pmatrix} 1 & 2 & 1 & 1\\ 
				2 & 2 & 2 & 1 \\ 
				1 & 2 & 2 & 2 \\ 
				1 & 1 & 2 & 1 
			\end{pmatrix}\in\Z^{4\times 4}.
		\end{align*}
		This matrix is totally $S(1)$-modular. It turns out that the unique way, up to a sign, to find the associate polynomials is given by the map that maps $\indet$ to $1$. 
		However, the then uniquely defined corresponding matrix
		\begin{align*}
			\polyM = \begin{pmatrix} \indet & \indet+1 & \indet & \indet\\ 
				\indet+1 & \indet+1 & \indet+1 & \indet \\ 
				\indet & \indet+1 & \indet+1 & \indet+1 \\ 
				\indet & \indet & \indet+1 & \indet 
			\end{pmatrix}\in\polyring^{4\times 4} 
		\end{align*}
		satisfies $\det\polyM = 1$ and is therefore not totally $\polyS$-modular.
	\end{example}
	
%	\section{The abstract setting and results}
%	
%	We make the remarks from the previous section precise:  
	
\subsection{Proof of Theorem~\ref{thm_from_integer_to_polyring}}
	For $S\subseteq \polyring$, let $\mathcal{M}(S)$ denote the set of all totally $S$-modular matrices. Recall that, given $\intvar\in\Z$, $\varphi_{\intvar}:\polyring\rightarrow \Z$ is the homomorphism which substitutes $x$ with $\intvar$, and $\Phi_{\intvar}:\mathcal{M}(\polyring)\rightarrow\mathcal{M}(\Z)$ is the function $\polyM\mapsto\polyM(\intvar)$ which applies $\varphi_{\intvar}$ entrywise to matrices in $\mathcal{M}(\polyring)$. We have by Lemma~\ref{lemma_totally_S_r_modular} that $\Phi_{\intvar}(\mathcal{M}(\polyS))\subseteq\mathcal{M}(\polyS(\intvar))$ for each $S\subseteq \polyring$.
	
	\begin{proof}[Proof of Theorem~\ref{thm_from_integer_to_polyring}]
		First, we show the claim for those $\intvar\in\Z$ with the property that $\varphi_{\intvar}$ defines a bijection between $\polyS$ and $S(\intvar)$. For such $\intvar\in\Z$, the map $\Phi_{\intvar}|_{\mathcal{M}(S)}$ is an injection. Hence, the statement boils down to showing that $\Ima\Phi_{\intvar}|_{\mathcal{M}(S)} = \mathcal{M}(S(\intvar))$ if and only if $\intvar \notin I(\polyS)$ for those values of $\intvar$. We will show the negated statement which is $\Ima\Phi_{\intvar}|_{\mathcal{M}(S)} \subsetneq \mathcal{M}(S(\intvar))$ if and only if $\intvar \in I(\polyS)$.
		
		Suppose there exists $\bN$ totally $S(\intvar)$-modular such that $\bN\notin \Ima\Phi_{\intvar}|_{\mathcal{M}(\polyS)}$. Since $\varphi_{\intvar}|_S$ is a bijection, we can construct a matrix $\polyM$ from $\bN$ by applying $(\varphi_{\intvar}|_S)^{-1}$ entrywise to $\bN$ so that $\bN = \bM(\intvar)$. Our assumption that $\bN\notin \Ima\Phi_{\intvar}|_{\mathcal{M}(\polyS)}$ translates to $\polyM$ not being totally $\polyS$-modular. So there exists a submatrix $\tilde{\polyM}$ of $\polyM$ which is a forbidden submatrix for $\polyS$. This implies $\det\tilde{\polyM}\in F(\polyS)$. However, by Lemma~\ref{lemma_determinant_homomorphism}, we get $\varphi_\intvar(\det\tilde{\polyM}) = \det\Phi_{\intvar}(\tilde{\polyM}) = \det\tilde{\bM}(\intvar)\in S(\intvar)$ since $\tilde{\bM}(\intvar)$ is a submatrix of $\bN = \bM(\intvar)$ which is totally $S(\intvar)$-modular. In other words, we have $(\det\tilde{\polyM})(\intvar) = s(\intvar)$ for some $s\in\polyS$, which implies $\intvar\in I(\polyS)$.
		
		Now suppose that $\intvar\in I(\polyS)$, that is, there exist $s\in \polyS$ and $f\in F(\polyS)$ such that $s(\intvar) = f(\intvar)$. Let $\polyM$ be a forbidden submatrix for $\polyS$ such that $\det\polyM = f$. Then, using Lemma~\ref{lemma_determinant_homomorphism}, we calculate that $\det\bM(\intvar) = \det\Phi_{\intvar}(\polyM) = \varphi_{\intvar}(f) = f(\intvar) = s(\intvar) \in S(\intvar)$. So $\bM(\intvar)$ is totally $S(\intvar)$-modular by Lemma~\ref{lemma_totally_S_r_modular}. Since $\varphi_{\intvar}|_S$ is a bijection, $\bM$ is the unique matrix with entries in $\polyS$ which maps to $\bM(\intvar)$ under $\Phi_{\intvar}$, so we conclude that $\bM(\intvar)\notin \Ima\Phi_{\intvar}|_{\mathcal{M}(\polyS)}$. This settles the claim for all $\intvar\in\Z$ that give rise to a bijection $\varphi_{\intvar}$ between $\polyS$ and $S(\intvar)$.
		
		We are left with the case when $\varphi_{\intvar}$ is not a bijection between $\polyS$ and $S(\intvar)$ for $\intvar\in\Z$. Fix such an $\intvar\in\Z$. By construction of $S(\intvar)$, the map $\varphi_{\intvar}$ is always surjective between $\polyS$ and $S(\intvar)$. So, to be not bijective, $\varphi_{\intvar}$ is not injective between $\polyS$ and $S(\intvar)$. This implies that there exist distinct $s_1,s_2\in \polyS$ such that $s_1(\intvar) = s_2(\intvar)$. Moreover, since $s_1$ and $s_2$ can be considered as $1\times 1$ matrices, we already know that $\Phi_{\intvar}|_{\mathcal{M}(S)}$ is not injective. Hence, it suffices to show that $\intvar\in I(\polyS)$. Suppose that $s_2 = 0$, i.e., $s_2$ is the zero-polynomial. There exists a minimal $k\in\N_{\geq 2}$ such that $s_1^k\notin S$. Then $\diag(s_1,\ldots,s_1)\in\polyring^{k\times k}$ is a forbidden submatrix for $\polyS$. The determinant of this matrix equals $s_1^k$ and satisfies $s_1^k(\intvar) = s_2^k(\intvar) = 0$, which gives us $\intvar \in I(\polyS)$ as $0\in\polyS$. Suppose now that neither $s_1$ nor $s_2$ is the zero-polynomial. Consider the matrix
		\begin{align*}
			\polyM = \begin{pmatrix}
				s_1 & s_2 \\ 
				s_ 2 & s_1 
			\end{pmatrix}\in\polyring^{2\times 2}.
		\end{align*}
		This matrix has determinant $s_1^2 - s_2^2$. If $s_1^2-s_2^2\notin S$, then $\polyM$ is a forbidden submatrix for $\polyS$ and $(s_1^2-s_2^2)(\intvar) = s_1^2(\intvar) - s_2^2(\intvar) = 0$ gives us $\intvar \in I(\polyS)$ as $0\in \polyS$. If $s_1^2-s_2^2\in S$, apply the same argument for $\tilde{s}_1 := s_1^2-s_2^2$ and $\tilde{s}_2 = 0$ from above to recover that $\intvar \in I(\polyS)$.
		
		To see the last statement, we have by assumption that $\polyS$ is finite, and hence $F(\polyS)$ is finite by Lemma~\ref{lemma_finite_number_of_determinants}. An integer $\intvar$ lies in $I(\polyS)$ if and only if it is an integral root of a polynomial of the form $s-f$ where $s\in\polyS$ and $f\in F(\polyS)$. Note that every such polynomial is nonzero, since $\polyS$ and $F(\polyS)$ are disjoint by definition of $F(\polyS)$. It follows that
		\[
		\left| I(\polyS)\right| \leq \left|S\right|\left|F(\polyS)\right|\cdot\max_{s\in \polyS\cup F(\polyS)}\deg(s).\qedhere
		\]
	\end{proof}
	
\subsection{Discussion with an Example}
		Let $\polyS = \pm\lbrace 0,1,\indet,\indet + 1,2\indet + 1\rbrace$. One can compute that $I(\polyS) = \lbrace -3,-2,-1,0,1,2\rbrace$; see Section~\ref{ss_computing_I(S)}. Hence, by Theorem~\ref{thm_from_integer_to_polyring} the evaluation map $\varphi_{\intvar}$ gives rise to an isomorphism for all integer values $\intvar\notin\lbrace -3,-2,-1,0,1,2\rbrace$. 
		
		Let $\intvar \in\lbrace -3,2\rbrace$. Then we have $S(\intvar) = \pm\lbrace 0,1,2,3,5\rbrace$. Observe that the map $\varphi_{\intvar}$ is a bijection between $\polyS$ and $S(\intvar)$. Theorem~\ref{thm_from_integer_to_polyring} implies that there exists a forbidden submatrix for $\polyS$ that evaluates to a totally $S(\intvar)$-modular matrix. Such a forbidden submatrix for $\polyS$ is given by
		\begin{align}
			\label{matrix_example_2}
			\polyM = \begin{pmatrix}
				x & 1 \\ 
				1 & x
			\end{pmatrix}\in\polyring^{2\times 2}.
		\end{align}
		This matrix has determinant $x^2-1$, which is forbidden. If $\intvar=2$ then $\det\bM(\intvar) = 3\in S(\intvar)$, so in this case the set of totally $S(\intvar)$-modular matrices is strictly richer than the set of totally $\polyS$-modular matrices since $\Ima\Phi_{\intvar}|_{\mathcal{M}(\polyS)}\subsetneq \mathcal{M}(S(\intvar))$ by Lemma~\ref{lemma_totally_S_r_modular}.
		
		Let $\intvar\in\lbrace -2,-1,0,1\rbrace$. Then we have 
		\begin{enumerate}
			\item $S(0) = \pm \lbrace 0,1\rbrace = S(-1)$ and
			\item $S(1) = \pm\lbrace 0,1,2,3\rbrace = S(-2)$.
		\end{enumerate}
		There is no bijection between the entries. This implies that there is not a unique way how to identify totally $S(\intvar)$-modular matrices with totally $\polyS$-modular matrices. When $\intvar = 0$, the matrix in (\ref{matrix_example_2}) is an example of a forbidden submatrix for $\polyS$ whose evaluation is totally $S(\intvar)$-modular. However, the matrix $\bM(0)$ remains an evaluation of some other totally $\polyS$-modular matrix since $\pm \lbrace 0,1\rbrace\subseteq \pm \lbrace 0,1,\indet,\indet + 1,2\indet + 1\rbrace$. The same argument implies that all totally $S(0)$-modular matrices arise as an evaluation of a totally $\polyS$-modular matrix. In other words, we have $\Ima\Phi_{\intvar}|_{\mathcal{M}(\polyS)} = \mathcal{M}(S(0))$ but not an injection. This has the consequence that our results for totally $\polyS$-modular matrices also imply results for totally $S(0)$-modular matrices.
		
		For $S(1) = \pm\lbrace 0,1,2,3\rbrace$, it would be interesting to know which matrices are evaluations of totally $\polyS$-modular matrices and which are not. The proof of Theorem~\ref{thm_from_integer_to_polyring} does not give an answer. It only guarantees the existence of a forbidden submatrix for $\polyS$ whose evaluated polynomial agrees with one of the evaluated polynomials in $S(1)$. The example (\ref{matrix_example_2}) with determinant $x^2-1$ is such a matrix but its evaluation can also be obtained by the all-ones matrix. A matrix that is totally $S(1)$-modular and is not an evaluation is given below
		\begin{align*}
			\begin{pmatrix}
				3 & 2 \\
				2 & 1
			\end{pmatrix}.
		\end{align*}
		The only matrices that evaluate to this matrix are
		\begin{align*}
			\begin{pmatrix}
				2\indet + 1 & \indet + 1\\ 
				\indet + 1 & 1
			\end{pmatrix}\text{ and }\begin{pmatrix}
				2\indet + 1 & \indet + 1\\ 
				\indet + 1 & \indet
			\end{pmatrix}.
		\end{align*}
		The determinants are in both cases quadratic so both matrices are forbidden submatrices for $\polyS$.
		
\subsection{Computing $I(\polyS)$}\label{ss_computing_I(S)}
	Following Theorem~\ref{thm_from_integer_to_polyring}, we demonstrate how to calculate all intersections between the polynomial functions in a given set $\polyS$ and the polynomial functions in $F(\polyS)$. To do so, we need to determine $F(\polyS)$ for a given $\polyS$. An initial superset of $F(\polyS)$ can be obtained by computing all polynomials $\det\polyM$ that satisfy the Desnanont-Jacobi identity, Lemma~\ref{lemma_desnanot_jacobi}. Call the set of those polynomials $DJ(\polyS)$. Since every element in $F(\polyS)$ has to arise as one of the solutions of the Desnanont-Jacobi identity, we have $F(\polyS)\subseteq DJ(\polyS)\backslash S \cap \polyring$. However, as we will see in the proof of Lemma~\ref{lemma_determinants_submatrices_1_a_a+1_2a+1}, this inclusion may not be tight, i.e., there could be polynomials in $DJ(\polyS)\backslash S \cap \polyring$ that do not correspond to an actual matrix. 
	To deal with these polynomials, we prove the following lemma which holds under more general assumptions and will also be used in the next section. 
%	For that purpose, we remark that the ring $\polyring$ is a unique factorization domain. So every element in $\polyring$ admits a unique factorization into smaller irreducible elements up to multiplying with $\pm 1$. Therefore, the notion of greatest common divisors carries naturally over from $\Z$ to $\polyring$. 
	Analogously to $\Z$, elements in $\polyring$ are relatively prime if their greatest common divisor is $1$.
	\begin{lemma}{\cite[Lemma~6]{celayakuhlmannweismantel2024matricespolynomial}}
	\label{lemma_equal_submatrices_tu}
		Let $\polyS \subseteq\polyring$ be such that all non-zero $y,z\in \polyS$ are relatively prime or $\left|y\right| = \left|z\right|$ and $2\notin \polyS$. Let $n\geq 3$ and $\polyM\in\polyring^{n\times n}$ be invertible and every $(n-1)\times (n-1)$ submatrix of $\polyM$ is totally $\polyS$-modular. Then either the invertible $(n-1)\times (n-1)$ submatrices of $\polyM$ are totally unimodular or there exist invertible submatrices $\bA_{(n-2)},\tilde{\bA}_{(n-2)}\subseteq\bM$ such that
		\begin{align*}
			\left|\det\bA_{(n-2)}\right|\neq\bigl|\det\tilde{\bA}_{(n-2)}\bigr|. 
		\end{align*}
	\end{lemma}
	\begin{proof}
		We prove the statement by induction. Let $n=3$. The $(n-2)\times (n-2)$ submatrices of $\polyM$ are the entries of $\polyM$. Assume that every entry of $\polyM$ is in $\pm\lbrace 0, s\rbrace$ for some $s\in S$. If $s=1$, we get that $\bM\in\pm\lbrace 0,1\rbrace^{3\times 3}$. Since $2\notin S$, the $2\times 2$ submatrices of $\bM$ are totally unimodular. If we assume $s\neq 1$, we have  $\det\bA_{(2)}\in\pm\lbrace s^2, 2s^2\rbrace$ for an invertible submatrix $\bA_{(2)}\subseteq\polyM$ as every entry is contained in $\pm\lbrace 0, s\rbrace$. Since the non-zero elements in $S$ are pairwise relatively prime and $s\in S$, we get $\det\bA_{(2)}\notin S$, contradicting that $\bA_{(2)}$ is totally $S$-modular. 
		
		So we suppose that $n\geq 4$. Assume that every invertible $(n-2)\times (n-2)$ submatrix of $\polyM$ has determinant $s$ or $-s$ for some $s\in S$. We show that every $(n-1)\times (n-1)$ submatrix of $\polyM$ is totally unimodular by a contradiction. For that purpose, suppose that $\bA_{(n-1)}\subseteq\polyM$ is an invertible $(n-1)\times (n-1)$ submatrix that is not totally unimodular. We apply the induction hypothesis to $\bA_{(n-1)}\subseteq\bM$. Since $\bA_{(n-1)}$ is not totally unimodular and $\left|\det\bA_{(n-1)}\right|\neq 2$, the induction hypothesis guarantees the existence of an invertible submatrix $\bA_{(n-3)}\subseteq\bA_{(n-1)}$ with $\bigl|\det\bA_{(n-3)}\bigr|\neq  \left|s\right|$. 
		We apply the Desnanot-Jacobi identity, Lemma~\ref{lemma_desnanot_jacobi}, to $\bA_{(n-3)}\subseteq \bA_{(n-1)}$ and obtain that 
		\begin{align*}
			\det\bA_{(n-1)} = \frac{1}{\det\bA_{(n-3)}} \left( s_1s_2-s_3s_4\right)\in\pm\frac{1}{\det\bA_{(n-3)}}\cdot\lbrace s^2, 2s^2\rbrace
		\end{align*}
		with $s_i\in\pm\lbrace 0,s\rbrace$ for $i=1,2,3,4$. If $\left|\det\bA_{(n-3)}\right| \neq 1$, then $\det\bA_{(n-3)}$ does not divide $s^2$ or $2s^2$ as the non-zero elements in $S$ are pairwise relatively prime and $2\notin S$. However, this implies that $\det\bA_{(n-1)}\notin \polyring$, a contradiction. If $\left|\det\bA_{(n-3)}\right| =1$, we get that $\det\bA_{(n-1)}\in \pm\lbrace s^2,2s^2\rbrace$. This gives us a contradiction since $\det\bA_{(n-1)}$ and $s$ are not relatively prime. 
	\end{proof}
	
	Let $F_2(\polyS)\subseteq F(\polyS)$ be the subset of determinants of forbidden submatrices for $\polyS$ of size $2\times 2$.  
	We are in the position to determine the set of possible determinants given by the forbidden submatrices for the specific set $\polyS=\pm\lbrace 0,1,\indet,\indet+1,2\indet+1\rbrace$.
	\begin{lemma}{\cite[Lemma~7]{celayakuhlmannweismantel2024matricespolynomial}}
		\label{lemma_determinants_submatrices_1_a_a+1_2a+1}
		Let $\polyS=\pm\lbrace 0,1,\indet,\indet+1,2\indet+1\rbrace$. Then we have $F(\polyS) = F_2(\polyS)$.
	\end{lemma}
	\begin{proof}
		Let $n\geq 3$ and $\polyM$ be a forbidden submatrix. We show that $\det\polyM \in F_2(\polyS)$. There exists an invertible submatrix $\bA_{(n-2)}\subseteq\polyM$. Applying the Desnanot-Jacobi identity, Lemma~\ref{lemma_desnanot_jacobi}, we get
		\begin{align}\label{eq_desnanot_jacobii}
			\det\polyM\cdot\det\bA_{(n-2)} = s_1s_2 - s_3s_4
		\end{align}
		with $s_i\in \polyS$ for $i=1,2,3,4$. This equation needs to have a solution for $\det\polyM$ in $\polyring$. Suppose that $\det\polyM$ has degree larger than one. This implies $\left|\det\bA_{(n-2)}\right| = 1$, an equality that has to hold for every invertible $(n-2)\times (n-2)$ subdeterminant. By Lemma~\ref{lemma_equal_submatrices_tu}, the invertible $(n-1)\times(n-1)$ submatrices of $\polyM$ are totally unimodular. Therefore, by \eqref{eq_desnanot_jacobii}, $\det\polyM$ is an integer which is a contradiction. So we assume $\det\bM$ has degree at most one. Using suitable software such as SageMath, one can enumerate all feasible solutions to the Desnanot-Jacobi identity with degree at most one. This gives us the following feasible values for $\det\polyM$ up to a sign:
		\begin{align*}
			2\polyS\backslash 0 \cup \lbrace \indet-1,\indet+2,2\indet-1,2\indet+3,3\indet+1,3\indet+2,4\indet,4\indet+1,4\indet+3,4\indet+4\rbrace.
		\end{align*}
		However, if $\det\polyM\in\pm\lbrace 2\indet-1,2\indet+3,4\indet,4\indet+1,4\indet+3,4\indet+4\rbrace$, then $\left|\det\bA_{(n-2)}\right|$ is uniquely determined in \eqref{eq_desnanot_jacobii} provided it is non-zero. Therefore, if there exists a forbidden submatrix with such a determinant, every invertible $(n-2)\times (n-2)$ submatrix needs to have the same determinant in absolute value. This contradicts Lemma~\ref{lemma_equal_submatrices_tu}. Hence, we can exclude these values. The remaining determinants are all contained in $F_2(\polyS)$. 
	\end{proof}
	The set $F_2(\polyS)$ can simply be computed by brute force. 
	Computing the intersections of the elements in $\polyS=\pm\lbrace 0,1,\indet,\indet+1,2\indet+1\rbrace$ and $F_2(\polyS)$ yields that $I(\polyS)=\lbrace -3,\ldots,2\rbrace$. 
	
	\begin{remark}
		\label{rem_intersections_1_x_x+1_2x+1}
		One can compute analogously to the proof of Lemma~\ref{lemma_determinants_submatrices_1_a_a+1_2a+1} the following sets $I(\polyS)$, which confirm the exceptions in Corollary~\ref{cor_recognition_int}.
		\begin{enumerate}
			\item For the set $\polyS = \pm \lbrace 0,1,\indet,\indet + 1\rbrace$, one has $F(S) = F_2(S)$ and $I(S) = \lbrace -3,\ldots,2\rbrace$.
			\item For the set $\polyS = \pm \lbrace 0,1,\indet\rbrace$, one has $F(S) = F_2(S)$ and $I(S) = \lbrace -2,\ldots,2\rbrace$.
		\end{enumerate}
	\end{remark}

	\section{The Recognition Problem}
	\label{sec_recognition}
	In this section we prove Theorem~\ref{thm_recognition_polyring}. We show how to recognize matrices with determinants in $\pm\lbrace 0,1,\indet,\indet + 1,2\indet + 1\rbrace$ and then, using this result, how to recognize subclasses of those matrices. We proceed as follows. Let $\polyS\subseteq \pm\lbrace 0,1,\indet,\indet + 1,2\indet + 1\rbrace$ be the given set of allowed polynomials. By Lemma~\ref{lemma_rank_coefficient_matrix}, every totally $\polyS$-modular matrix $\polyM$ admits a decomposition of the form $\polyM = \bM(0) + \indet\cdot\bm{u}\cdot\bm{v}^\top$ for some $\bu$ and $\bv$. One can test in polynomial time whether such a decomposition exists. For the sets $\polyS$ to be considered below, we use a specific characterization of when $\polyM$ is totally $\polyS$-modular. In all these cases, the characterization of total $\polyS$-modularity involves testing whether certain matrices are totally unimodular.	The polynomial running time follows then from the fact that one can recognize totally unimodular matrices in polynomial time; see \cite[Chapter 20]{schrijvertheorylinint86}.
	
	\subsection{The case $\pm \lbrace 0,1,\indet,\indet + 1,2\indet + 1\rbrace$} 
	For $S=\pm\lbrace 0,1,\indet,\indet+1,2\indet+1\rbrace$, 
%	By Lemma~\ref{lemma_rank_coefficient_matrix}, every totally $\polyS$-modular matrix $\polyM$ admits a decomposition into $\polyM = \bM(0) + \indet\cdot\bm{u}\cdot\bm{v}^\top$ for some $\bu$ and $\bv$.
%	One can test in polynomial time whether such a decomposition exists. 
	we use the complete characterization given below. 
	\begin{lemma}
		\label{lemma_t_bar_t_tu}
		Let $S=\pm\lbrace 0,1,\indet,\indet+1,2\indet+1\rbrace$. A matrix $\polyM = \bM(0) + \indet \cdot\bm{u}\cdot\bm{v}^\top$ is totally $\polyS$-modular if and only if $\bM(0)$ and $\bM(-1)$ are totally unimodular.
	\end{lemma}
	\begin{proof}
		It suffices to assume $\bM$ is square. Suppose $\bM(0)$ and $\bM(-1)$ are totally unimodular. We have
		\begin{align}\label{eq_proof_recognition_matrix_determinant_lemma}
			\det\polyM = \left(\det\bM(0) - \det\bM(-1)\right)\cdot \indet + \det\bM(0)
		\end{align}
		by Lemma \ref{lemma_matrix_determinant_lemma}. So $\det\polyM$ is completely determined by $\det\bM(0)$ and $\det\bM(-1)$. 
		As $\det\bM(0)$ and $\det\bM(-1)$ are both in $\pm\lbrace 0, 1\rbrace$, we obtain that $\det\polyM\in\polyS$ by going through all feasible cases. The other direction follows directly by evaluating $\polyS$ and $\polyM$ at $\indet\in\lbrace -1,0\rbrace$. 
	\end{proof}
	
	\subsection{The case $\pm \lbrace 0,1,\indet,\indet + 1\rbrace$}
	Let $\polyS = \pm \lbrace 0,1,\indet,\indet + 1\rbrace$. 
%	As before, by Lemma~\ref{lemma_rank_coefficient_matrix}, every totally $\polyS$-modular matrix $\polyM$ admits a decomposition into $\polyM = \bM(0) + \indet \cdot \bu\cdot\bv^\top$ for some vectors $\bu,\bv$ and one can test whether such a decomposition exists in polynomial time. 
	We aim to give a characterization similar as Lemma~\ref{lemma_t_bar_t_tu}. For that purpose, we introduce the concept of slope matrices. Let $\polyM = \bM(0) + \indet\cdot\bu\cdot\bv^\top$. Then the slope matrix of $\bM$ is defined by 
	\begin{align*}
		\slopematrix(\bM) := \begin{pmatrix}
			0 & \bv^\top \\ 
			\bu & \bM(0)
		\end{pmatrix}.
	\end{align*} 
	Given that $\bM$ is square, we denote by $\slope(\bM)$ the determinant of  $\slopematrix(\bM)$. It can be shown using the matrix determinant lemma and applying Laplace expansion along the first row in $\slopematrix(\bM)$ and then on the first column of the resulting matrices that
	\begin{align}
		\label{eq_matrix_determinant_lemma}
		\det \polyM = \det\bM(0) - \indet\cdot\slope(\bM).
	\end{align}
	We prove in Lemma~\ref{lemma_recognition_characterization_1_x_x+1} an analogue of Lemma~\ref{lemma_t_bar_t_tu} that characterizes totally $\polyS$-modular matrices by checking if $\polyM(0)$, $\polyM(-1)$, and $\slopematrix(\polyM)$ are totally unimodular, which again can be done in polynomial time. To deal with slope matrices we need the following intermediate result which uses Lemma~\ref{lemma_equal_submatrices_tu}.
	
	\begin{lemma}
		\label{lemma_append_u_v_1_x_x+1}
		Let $\polyS = \pm \lbrace 0,1,\indet,\indet + 1\rbrace$, and let $\polyM = \bM(0) + \indet \cdot\bu\cdot\bv^\top$ be an $m\times n$ totally $\polyS$-modular matrix where $\bu\neq\bm{0}$ and $\bv\neq\bm{0}$. Then $(\bM(0) \; \bu)$ is totally unimodular.
	\end{lemma}
	
	\begin{proof}
		It suffices to consider the case $m,n\geq 2$. The matrix $\bM(0)$ is totally unimodular by definition and $\bu$ and $\bv$ have entries in $\pm\lbrace 0,1\rbrace$. 
		Suppose $(\bM(0) \; \bu)$ is not totally unimodular. So the matrix contains a forbidden submatrix for $\pm\lbrace 0,1\rbrace$. Since $\bM(0)$ is totally unimodular, a forbidden submatrix has to meet the column $\bu$. Let $(\bM(0)_{I,J} \; \bu_I)$ be such a forbidden submatrix indexed by $I\subseteq \lbrack m\rbrack$ and $J\subseteq \lbrack n\rbrack$ with $|I|=k$ and $|J|=k-1$. We can add or subtract $\indet\cdot\bu_I$ to columns of $\bM(0)_{I,J}$ such that we obtain the matrix $(\bM(0)_{I,J}+ \indet\cdot\bu_I\cdot\bv_J^\top \quad \bu_I) = (\polyM_{I,J} \; \bu_I)$. This does not alter subdeterminants that meet the column $\bu_I$. 
		Since $\polyM_{I,J}$ is a submatrix of $\polyM$, each proper submatrix of $(\polyM_{I,J}\; \bu_I)$ is totally $\polyS$-modular. 
		
		We claim that all invertible $(k-2)\times (k-2)$ submatrices of $\polyM_{I,J}$ have the same determinant in absolute value. To this end, we show that if $\bA_{(k-2)}$ is an invertible $(k-2)\times(k-2)$ submatrix of an invertible $(k-1)\times(k-1)$ submatrix $\bA_{(k-1)}$ of $\bM_{I,J}$, then $|\det\bA_{(k-2)}	|=|\det\bA_{(k-1)}|$. The claim follows since every invertible $(k-2)\times(k-2)$ submatrix of $\polyM_{I,J}$ is contained in an invertible $(k-1)\times(k-1)$ submatrix of $\polyM_{I,J}$, and any two invertible $(k-1)\times(k-1)$ submatrices of $\polyM_{I,J}$ share at least one invertible $(k-2)\times(k-2)$ submatrix.
			
		Let $\bA_{(k-2)}\subseteq \bA_{(k-1)}\subseteq \polyM_{I,J}$ be invertible $(k-2)\times (k-2)$ and $(k-1)\times (k-1)$ submatrices, respectively. Then, by Lemma~\ref{lemma_desnanot_jacobi}, we recover that 
		\begin{align*}
			2\cdot\det\bA_{(k-2)} = t_1\cdot\det\bA_{(k-1)} - t_2\cdot s
		\end{align*}
		for $s\in \polyS$ and $t_1,t_2\in \pm \lbrace 0,1\rbrace$ as $t_1,t_2$ correspond to the $(k-1)\times (k-1)$ submatrices that meet $\bu_I$, which admit determinants contained in $\pm \lbrace 0,1\rbrace$. If one of the $t_i$ equals zero, the Desnanot-Jacobi identity above has no solution. So we must have $t_1,t_2\in\{-1,1\}$. The only possibility to satisfy this equation is given by $\left|s\right| = \left|\det\bA_{(k-2)}\right| = \left|\det\bA_{(k-1)}\right|$.
		
		Denote by $s$ the common absolute value of the non-zero $(k-2)\times(k-2)$ subdeterminants of $\polyM_{I,J}$. We scale the last column and analyze $(\polyM_{I,J} \;\, s\cdot \bu_I)$. The invertible $(k-2) \times (k-2)$ submatrices that meet $s\cdot\bu_I$ admit a determinant $\pm s$. So each $(k-2)\times (k-2)$ submatrix of $(\polyM_{I,J} \;\, s\cdot \bu_I)$ has a determinant contained in $\pm\lbrace 0,s\rbrace$. 
		By Lemma~\ref{lemma_equal_submatrices_tu}, it follows that the invertible $(k-1)\times (k-1)$ submatrices of $(\polyM_{I,J} \;\, s\cdot \bu_I)$ are totally unimodular. This implies $s = 1$ and $\polyM_{I,J}$ has entries in $\pm\lbrace 0,1\rbrace$. Since $\polyM_{I,J}=\polyM_{I,J}(0)+x\cdot\bu_I\cdot\bv_J^\top$, we have $\bu_I\cdot\bv_J^\top=\bm{0}$. 
		To satisfy this equation, we need to have $\bv_J = \bm{0}$ as $\bu_I = \bm{0}$ would give us $\det(\bM(0)_{I,J}\;\bu_I) = 0$. Let $j\in\lbrack n\rbrack$ be such that the column $\polyM_{I,j}$  contains some entries in $\pm \lbrace \indet,\indet + 1\rbrace$. Such a column exists because $\bv\neq\bm{0}$. By the multilinearity of the determinant, we obtain
		\begin{align*}
			\left|\det\polyM_{I, J\cup j}\right| 
			&= \left|\det(\bM(0)_{I,J} \; \polyM_{I,j})\right| \\
			&= \left|\det\bM(0)_{I, J\cup j} \pm \indet\cdot\det (\bM(0)_{I,J} \; \bu_I)\right| \\
			&= \left|t + 2\cdot\indet\right|
		\end{align*}
		for $t\in\pm \lbrace 0, 1\rbrace$, where we use that $(\bM(0)_{I,J} \; \bu_I)$ is a forbidden submatrix for $\pm\lbrace 0,1\rbrace$ and has therefore determinant $\pm 2$. However, this is a contradiction of the fact $\det\polyM_{I, J\cup j}\in \polyS$, since $\polyS$ has no linear form with leading coefficient $\pm 2$. We conclude that $(\bM(0)\;\bu)$ is totally unimodular.
	\end{proof}
	
	\begin{lemma}
		\label{lemma_recognition_characterization_1_x_x+1}
		Let $\polyS = \pm\lbrace 0,1,\indet,\indet + 1\rbrace$. A matrix $\polyM = \bM(0) + \indet \cdot\bu\cdot\bv^\top$ is totally $\polyS$-modular if and only if $\bM(0)$, $\bM(-1)$, and $\slopematrix(\bM)$ are totally unimodular.
	\end{lemma}
	\begin{proof}
		It suffices to show the statement for square matrices. Assume $\bM(0)$, $\bM(-1)$, and $\slopematrix(\bM)$ are totally unimodular. 
		Following Lemma~\ref{lemma_t_bar_t_tu}, we deduce that $\polyM$ is totally $\pm\lbrace 0,1,\indet,\indet + 1, 2\indet + 1\rbrace$-modular since $\bM(0)$ and $\bM(-1)$ are totally unimodular. We argue that no submatrix admits a determinant contained in $\pm\lbrace 2\indet + 1\rbrace$. Let $\polyM'\subseteq \polyM$ be a square submatrix. Then observe that $\slopematrix(\bM')$ is a submatrix of $\slopematrix(\bM)$. So we deduce that $\slope(\bM') \in \pm\lbrace 0,1\rbrace$ as $\slopematrix(\bM)$ is totally unimodular. By \eqref{eq_matrix_determinant_lemma}, we obtain that no square submatrix of $\bM$ has a determinant contained in $\pm\lbrace 2\indet + 1\rbrace$. So the matrix $\polyM$ is totally $\polyS$-modular. 
		
		Now suppose $\polyM$ is totally $\polyS$-modular. By evaluating at $\intvar\in\lbrace -1,0\rbrace$, the matrices $\bM(0)$ and $\bM(-1)$ are totally unimodular. Every square submatrix of $\slopematrix(\bM)$ that meets the first row and column has a determinant in $\pm \lbrace 0,1\rbrace$. This is because such matrices are slope matrices of totally $\polyS$-modular matrices, which have a slope of $\pm \lbrace 0,1\rbrace$ by (\ref{eq_matrix_determinant_lemma}). Also, the square submatrices of $\slopematrix(\bM)$ that do not meet both the first row and column are also totally unimodular as these are submatrices of $\bM(0)$. It remains to argue that submatrices of $\slopematrix(\bM)$ that meet the first row but not the first column, or the first column but not the first row, are totally unimodular as well. But this follows from Lemma~\ref{lemma_append_u_v_1_x_x+1} applied to $(\bM(0) \; \bu)$ and $(\bM(0)^\top \; \bv)$ .
	\end{proof}
	
	\subsection{The case $\pm\lbrace 0,1,\indet\rbrace$}
	For $\polyS = \pm \lbrace 0,1,\indet\rbrace$, 
%	As before, we begin by testing in polynomial time whether the decomposition $\polyM = \bM(0) + \indet\cdot\bu\cdot\bv^\top$ for some $\bu$ and $\bv$ applies. 
	one can show the following characterization similar to Lemmas~\ref{lemma_t_bar_t_tu} and \ref{lemma_recognition_characterization_1_x_x+1}.
	\begin{lemma}
		\label{lemma_recognition_characterization_0_1_x}
		Let $\polyS = \pm\lbrace 0,1,\indet\rbrace$. A matrix $\polyM = \bM(0) + \indet \cdot\bu\cdot\bv^\top$ is totally $\polyS$-modular if and only if $\bM(0)$, $\bM(-1)$, and $\bM(1)$ are totally unimodular.
	\end{lemma}
	\begin{proof}
		It suffices to show the statement for square matrices. Let $\bM(0)$, $\bM(-1)$, and $\bM(1)$ be totally unimodular. By the matrix determinant lemma, we have 
		\begin{align*}
			\det\bM = (\det\bM(0) - \det\bM(-1))\cdot\indet + \det\bM(0).
		\end{align*}
		Evaluating at $\indet = 1$ gives us
		\begin{align*}
			\det\bM(1) = 2\det\bM(0) - \det\bM(-1).
		\end{align*}
		Since $\det\bM(-1),\det\bM(0),\det\bM(1)\in\pm\lbrace 0,1\rbrace$, we can enumerate all feasible cases that satisfy the equation above and obtain $\det\polyM\in\pm\lbrace 0,1,\indet\rbrace$.
		
		For the other direction, suppose that $\polyM = \bM(0) + \indet \cdot\bu\cdot\bv^\top$ is totally $\polyS$-modular. The claim follows by evaluating at $\intvar\in\lbrace -1,0,1\rbrace$.
	\end{proof}
	
	\section{The Optimization Problem}
	\label{sec_optimization}
		We now turn to the problem of integer linear optimization when the constraint matrix is totally $\polyS$-modular. For a matrix $\bM$ and a row vector $\bv^\top$, let $\bM_{\bv}$ denote the matrix obtained by adding the row vector $\bv^\top$ to the bottom of $\bM$.
	
	\begin{lemma}
		\label{lemma_bounded_subdeterminants}
		Let $S\subseteq \polyring$ be finite and $\max_{s\in S} \deg(s) = 1$. Let $\polyM = \bM(0) + \indet\cdot\bu\cdot\bv^\top$ be totally $\polyS$-modular. Suppose that $\bu,\bv\neq \bm{0}$. Then  
		\begin{align*}
			\Delta\left(\bM(0)_{\bv}\right)\leq 2\cdot\Delta(\bM(0))\cdot\Delta(\bM(-1)).
		\end{align*}
	\end{lemma}
	\begin{proof}%S: We do not seem to use that M has full rank. i removed this assumption
		Define
		\[
		\left(\bp\ \bq\ \bT\right):=\left(\begin{array}{cccc}
			1 & 1 & \bu^{\top} & \zero \\
			\zero & \bv & \bM(0)^{\top} & \id
		\end{array}\right).
		\]	
		The matrix $\bT$ has full row rank since $\bu\neq \bm{0}$ and $\bp,\bq$ are non-parallel as $\bv\neq \bm{0}$. So we can apply Lemma~\ref{lemma_bound_subdeterminants_matrix_projection} and the result follows, noting that
		\begin{align*}
			\Delta(\bM(0))&=\Delta_n(\bT/\bp),\\
			\Delta(\bM(-1))&=\Delta_n(\bT/\bq),\\
			\Delta(\bM(0)_{\bv})&=\Delta_n((\bq\;\bT)/\bp)\leq\Delta_{n+1}(\bp\;\bq\;\bT).
		\end{align*}
	\end{proof}
	
	One can interpret this upper bound in terms of polynomials in $\polyS$, leveraging the fact that
	\begin{align*}
		\Delta\left(\bM(0)\right) \leq \max_{\alpha_1\indet+\alpha_0\in S}|\alpha_0| \text{ and }\Delta\left(\bM(-1)\right) \leq \max_{\alpha_1\indet+\alpha_0\in S}|\alpha_1-\alpha_0|.
	\end{align*}
	Using Lemma~\ref{lemma_bounded_subdeterminants} we can show the following reduction from $\ILP(\bM(\intvar),\bb,\bc)$ to integer linear programs with constraint matrix $\bM(0)_{\bv}$.
	
	\begin{lemma}
		\label{lemma_reduction}
		Let $S\subseteq \polyring$ be finite and $\max_{s\in S} \deg(s) = 1$. Let $\polyM = \bM(0) + \indet\cdot\bu\cdot\bv^\top$ be totally $\polyS$-modular and $\intvar\in\Z$ such that $\bM(\intvar)$ has full column rank. Then one can solve $\ILP(\bM(\intvar),\bb,\bc)$ in polynomial time if one can solve
		\begin{align*}
			\max \bc^\top\bx, \text{ s.t. } \bM(0)\bx\leq\tilde{\bb}, \ \bv^\top\bx = y, \ \bx\in\Z^n
		\end{align*}
		in polynomial time for all $\tilde{\bb}\in\Z^{m}$ and $y\in\Z$.
	\end{lemma}
	
	Using these two lemmas, we can prove Theorem~\ref{thm_opt_polyring}.
	
%	\begin{lemma}
%		\label{lemma_linear_independence_evaluation}
%		Let $S\subseteq \polyring$ be finite and $\max_{s\in S} \deg(s) = 1$. Let $\polyM$ be totally $\polyS$-modular. Suppose that $\bM(\intvar)$ has full column rank over $\Z$ for some $\intvar\in\Z$. Then $\polyM$ has full column rank over $\polyring$.
%	\end{lemma}

	\begin{proof}[Proof of Theorem~\ref{thm_opt_polyring}]
		Suppose that $\bu\neq \bm{0}$ and $\bv\neq \bm{0}$, otherwise $\polyM$ is totally unimodular and the claim follows. It suffices to be able to solve integer linear programs with constraint matrix $\bM(0)_{\bv}$ by Lemma~\ref{lemma_reduction}. We can bound the largest determinant of $\bM(0)_{\bv}$ using Lemma~\ref{lemma_bounded_subdeterminants}. %since $\polyM$ has full column rank by Lemma~\ref{lemma_linear_independence_evaluation}. 
		This gives us $\Delta\left(\bM(0)_{\bv}\right)\leq 2$. This implies the claim since integer linear programs with totally bimodular constraint matrix can be solved in polynomial time; see \cite{artweiszenbimodalgo2017}.
	\end{proof}
	
%	\subsection{Results on linear independence}
%	We will use the following results that follow from the matrix determinant lemma, Lemma~\ref{lemma_matrix_determinant_lemma}.
%	\begin{lemma}
%		\label{lemma_linear_independence_equivalence}
%		Let $S\subseteq \polyring$ be finite and $\max_{s\in S} \deg(s) = 1$. Let $\polyM$ be totally $\polyS$-modular. Then $\polyM$ has full column rank over $\polyring$ if and only if at least one of $\bM(0)$ or $\bM(-1)$ has full column rank over $\Z$.
%	\end{lemma}
	
	\subsection{Proof of Lemma~\ref{lemma_reduction}}
	We substitute $y := \bv^\top\bx$ in $\ILP(\bM(\intvar),\bb,\bc)$, which results in  
	\begin{align*}
		\max \bc^\top\bx, \text{ s.t. } \bM(0)\bx + \intvar\cdot\bu \cdot y \leq \bb, \ \bv^\top\bx = y, \ \bx\in\Z^n.
	\end{align*}
	Since we only care for integral solutions $\bx$, we can additionally assume that $y = \bv^\top\bx$ is an integer as $\bv$ is an integral vector. This allows us to treat $y$ as an additional integer valued variable. So the resulting integer linear program is defined by the inequality system
	\begin{align}
		\label{inequ_proof_lemma_two}
		\underbrace{
			\begin{pmatrix}
				\bM(0) & \intvar\cdot\bu \\
				\phantom{-}\bv^\top & -1 \\
				-\bv^\top & \phantom{-}1
		\end{pmatrix}}_{=:\bD}
		\begin{pmatrix}
			\bx \\ y
		\end{pmatrix}
		\leq \begin{pmatrix}
			\bb \\ 0 \\ 0
		\end{pmatrix}.
	\end{align}
	We next argue that optimal integral solutions can only attain certain values of $y$. Let $P\subseteq \R^{n+1}$ denote the feasible region given by the inequality system (\ref{inequ_proof_lemma_two}). Observe that $\bD$ has full column rank. This is indeed true, because $(n+1)\times (n+1)$ subdeterminants of $\bD$ that contain the last row are equal to $n\times n$ subdeterminants of $\bM(\intvar)$, which has full column rank by assumption. 
	
	Suppose there exists an optimal solution $(\bx^*,y^*)^\top$ to $\max_{(\bx,y)^\top\in P} \bc^\top\bx$.
	Since $\bD$ has full column rank, we may further assume that this solution is a vertex. 
	This vertex can be computed in polynomial time by solving a linear program. 
	We want to argue that there is an optimal integer solution of $\max_{(\bx,y)^\top\in P} \bc^\top\bx$ close to $(\bx^*,y^*)^\top$, where the distance is measured in the variable $y$ or, equivalently with respect to the $(n+1)$\textsuperscript{st} unit vector $\be^{n+1}$. We apply the proximity bound by Cook et al. \cite{CGST1986} in its more general form as outlined in \cite{celayakuhlpaatweis2023proxandflatness} with $\balpha = \be^{n+1}$. That guarantees the existence of some optimal integer solution $(\bz^*,\alpha^*)^\top\in P$ with  
%	We may further assume this solution is a vertex, because $\polyM$ has full column rank and full rank subdeterminants that contain the last row are equal to full rank subdeterminants of $\bM(\intvar)$. The full rank property follows then from Lemma~\ref{lemma_linear_independence_evaluation}. 
%	This vertex can be computed in polynomial time by solving a linear program. 
%	We want to argue that there is an optimal integer solution of $\max_{(\bx,y)^\top\in P} \bc^\top\bx$ close to $(\bx^*,y^*)^\top$, where the distance is measured in the variable $y$ or, equivalently with respect to the $(n+1)$\textsuperscript{st} unit vector $\be^{n+1}$. We apply the proximity bound by Cook et al. \cite{CGST1986} that guarantees the existence of some optimal integer solution $(\bz^*,\alpha^*)^\top\in P$ with  
	\begin{align*}
		\left|y^* - \alpha^*\right| = |(\bx^*,y^*)^\top\be^{n+1} - (\bz^*,\alpha^*)^\top\be^{n+1}| < n\cdot \Delta_n(\bD/(\be^{n+1})^\top)=n\cdot\Delta_n(\bM(0)_{\bv})
	\end{align*}
	provided that $P\cap\Z^{n+1}\neq \emptyset$. Hence, it suffices to check integer linear programs where $y\in \Z$ is restricted to integers that are contained in the interval 
	\begin{align}
		\label{proof_opt_interval}
		y^* + \left[- n\cdot\Delta_n\left(\bM(0)_{\bv}\right),n\cdot\Delta_n\left(\bM(0)_{\bv}\right)\right].
	\end{align}
Since $\bM(0)_{\bv}$ has bounded subdeterminants, see Lemma~\ref{lemma_bounded_subdeterminants}, there are only $\mathcal{O}(n)$ many choices of $y$. Therefore we can solve $\ILP(\bM(\intvar),\bb,\bc)$ by solving $\mathcal{O}(n)$ integer linear programs of the type given in the statement.

It remains to consider the case where $\max_{(\bx,y)^\top\in P} \bc^\top\bx$ is unbounded. If we are in this situation, then, as the constraint matrix $\bD$ has full column rank, we can apply the Cook et al. bound as before. Specifically, given \emph{any} feasible point $(\bx^*,y^*)^\top\in P$, we solve one integer linear program for each choice of $y\in \Z$ in the interval \eqref{proof_opt_interval}. 
If a solution $(\bz^*,\alpha^*)^\top\in P\cap\Z^{n+1}$ is found, then we know that our original integer linear program is unbounded. This is because we are in the unbounded setting, so there exists an integer vector $\br$ for which $(\bz^*,\alpha^*)^\top+\lambda\br\in P$ for all $\lambda\geq 0$, and $\bc^\top\br>0$. Otherwise, if no such solution $(\bz^*,\alpha^*)^\top$ is found, then we conclude that $P\cap\Z^{n+1}= \emptyset$.
\section{A Connection to Bimodular Matrices}
\label{sec_bimodular}
Theorem~\ref{thm_bimodular_matrix_projection} relates the family of $\pm\{0,1,x,x+1,2x+1\}$-modular matrices to bimodular matrices with at least two unimodular projections. We now give a proof of this theorem, and explore one example of a matrix from this family.

\subsection{Proof of Theorem~\ref{thm_bimodular_matrix_projection}}

\begin{lemma}
\label{lemma_matrix_as_projection}
Let $\bA$ be a full row rank $\polyS$-modular matrix, where $S\subseteq \polyring$ consists of linear forms only. Assume that at least one maximal subdeterminant of $\bA$ is not an integer, and that the greatest common divisor of the maximal subdeterminants of $\bA$ is $1$. Let $B$ be a basis of $\bA$ such that $s:=cx+d:=\det\bA_B$ and $\gcd(c,d)=1$. Let $a,b\in\Z$ so that $ad-bc=1$, and let $s'=ax+b$. Then there exists a full row rank integer matrix $\bT$ and non-parallel, non-zero integral vectors $\bp,\bq$ such that $\bA=\bT/(s'\bp+s\bq)$. 
\end{lemma}

\begin{proof}
First, note that $\bA_B^{-1}\bA$ is $\polyS'$-modular where
\[
\polyS':=\left\{\frac{a'x+b'}{cx+d}:a'x+b'\in\polyS\right\}.
\]
Let $\hat{\bA}=(\bA_B^{-1}\bA)(-\frac{dx-b}{cx-a})$. Since $ad-bc=1$, we have that $\hat{\bA}$ is $\polyS''$-modular where
\[
\polyS'':=\left\{(a'd-b'c)x+(ab'-a'b):a'x+b'\in\polyS\right\}\subseteq\polyring.
\]
In fact $\hat{\bA}$ is totally $\polyS''$-modular, since $\hat{\bA}_B$ is an identity matrix. So we may apply Lemma~\ref{lemma_rank_coefficient_matrix} and obtain $\hat{\bA}=\hat{\bA}(0)+x\bu\bv^\top$ for integral $\bu$ and $\bv^\top$. Note that $\hat{\bA}(0)$ is an integral matrix.

We show both $\bu$ and $\bv^\top$ are non-zero. We have $\bA^{-1}_B\bA=\hat{\bA}(s'/s)$, and therefore $\bA=\bA_B(\hat{\bA}(s'/s))$. If either $\bu$ or $\bv^\top$ are zero then $\hat{\bA}=\hat{\bA}(0)$, hence $\bA=\bA_B(\hat{\bA}(0))$. If $s=1$, we obtain a contradiction of the assumption that at least one maximal subdeterminant of $\bA$ is not an integer. If $s\neq 1$, then $s$ divides every maximal subdeterminant of $\bA$, but this is a contradiction of the assumption that the greatest common divisor of the maximal subdeterminants of $\bA$ is 1. We conclude $\bu$ and $\bv^\top$ must both be non-zero.

Define the integer matrix
\[
\left(\bp\ \bq\ \bT\right):=\left(\begin{array}{ccc}
0 & 1 & -\bv^{\top}\\
\bu & \zero & \hat{\bA}(0)
\end{array}\right).
\]
Now, $\bu$ is non-zero, which implies $\bv^\top_B$ is zero as $\hat{\bA}_B$ is an identity matrix. As $\bv^\top$ is non-zero we conclude that $\bT$ has full row rank. Because $\bu\neq\zero$, we have that $\bp$ and $\bq$ are non-zero and non-parallel. By construction we have $\hat{\bA}=\bT/(x\bp + \bq)$. Hence by Lemma~\ref{lemma_matrix_projection_variable_substitution} we have
\[
\bA^{-1}_B\bA=\hat{\bA}(s'/s)=\bT/((s'/s)\bp + \bq),
\]
and from this we conclude $\bA=\bT/(s'\bp + s\bq)$.
\end{proof}

%\begin{proposition}We have $\bA$ is $S'$-modular if and only if there exists a full row rank bimodular matrix $\bT$ and non-zero, non-parallel integer vectors $\bu,\bw$ such that $\bT/(\bu-\bw)$, $\bT/\bu$, and $\bT/(\bu+\bw)$ are unimodular and $\bA=\bT/(x\bu+(x+1)\bw)$.
%\end{proposition}
%Now suppose $\bA$ is $\polyS'$-modular. It remains to show that $\bA$ is trivially $\polyS'$-modular. For this it suffices to show that the matrix $\bT':=[\bu\ \bv\ \bT]$ is unimodular, and that there is no circuit of $\bT'$ (i.e. a minimal-support element of the kernel of $\bT'$) whose support contains the first two columns of $\bT'$. 

%By the above equation, we must have $\det(\bu\ \bT_I)\det(\bv\ \bT_I)=0$ for all $I$.

\begin{proof}[Proof of Theorem~\ref{thm_bimodular_matrix_projection}]
 Suppose $\bA$ is $S$-modular. Let $s:=cx+d\in\polyS$ be non-zero. Since $\gcd(c,d)=1$, there exist integers $a,b$ such that $ad-bc=1$. Let $s'=ax+b$. 
 By Lemma~\ref{lemma_matrix_as_projection} we have that there exists a full row rank matrix $\bT$ and non-zero, non-parallel integral vectors $\bp',\bq'$ such that $\bA=\bT/(s'\bp'+s\bq')$. Now let $\bp=(a-b)\bp'+(c-d)\bq'$ and $\bq=b\bp'+d\bq'$. Then $\bp,\bq$ are non-zero, non-parallel integral vectors such that $\bA=\bT/(x\bp + (x+1)\bq)$. By Lemma~\ref{lemma_matrix_projection_variable_substitution}, we have $\bA(0)=\bT/\bq$ and $\bA(-1)=\bT/(-\bp)$, and since $\polyS(0)=\polyS(-1)=\{-1,0,1\}$, we get by Lemma~\ref{lemma_totally_S_r_modular} that $\bT/\bp$ and $\bT/\bq$ are unimodular. By Lemma~\ref{lemma_bound_subdeterminants_matrix_projection}, then, $\bT$ is a bimodular matrix.
 
Conversely, suppose $\bT$ is a full row rank bimodular matrix and $\bp,\bq$ are non-zero, non-parallel so that $\bT/\bp$ and $\bT/\bq$ are unimodular. Let $\bA=\bT/(x\bp+(x+1)\bq)$. Let $I$ be a set of $m$ columns of $\bA$, where $m$ is the number of rows of $\bA$. We have
\[
\det(\bA_I)=\det(x\bp+(x+1)\bq\quad\bT_I)=x\det(\bp\ \bT_I)+(x+1)\det(\bq\ \bT_I).
\]
We have $\det(\bp\ \bT_I),\det(\bq\ \bT_I)\in\{-1,0,1\}$, which yields precisely 9 possible values for $\det(\bA_I)$. One can check that these values are precisely the elements of $S$.
\end{proof}

\subsection{An Example}
Let $\polyS=\pm\{0,1,x,x+1,2x+1\}$. We finish this section with an example to show that the $\polyS$-modular matrices can be understood as a family of matrices which sits properly between the unimodular matrices and bimodular matrices.

Consider the matrix
\[
\bA:=\begin{pmatrix}
1 & 1 & 1 &   &   &   & \phantom{-}1 & \phantom{-}1 & \phantom{-}1 &   &   &   & 1 \\
1 &   &   & 1 & 1 &   & -1 &   &   & \phantom{-}1 & \phantom{-}1 &   & 1 \\
  & 1 &   & 1 &   & 1 &   & -1 &   & -1 &   & \phantom{-}1 & 1 \\
  &   & 1 &   & 1 & 1 &   &   & -1 &   & -1 & -1 & 1
\end{pmatrix}.
\]
This matrix is $\pm\{0,2,4\}$-modular. Hence, if $B$ is a basis of $\bA$ for which $\det\bA_B=\pm 2$, then the matrix $\bT:=\bA_B^{-1}\bA$ is bimodular.

It can be shown that for all non-zero $\bp\in\Z^4$, the matrix $(\bT\; \bp)$ is bimodular if and only if $\bp$ already appears up to sign as a column in $\bT$. This means that $\bT/\bp$ can only be unimodular if $\bp$ already appears up to sign as a column in $\bT$. However, one can verify with a straightforward computation that there is no column $\bp$ of $\bT$ for which $\bT/\bp$ is unimodular. We conclude that $\bT$ is an example of a bimodular matrix which does not admit even a single unimodular projection, let alone two. 

Now consider the columns
\[
\left(\bp\; \br\; \bq\; \bs \right)=\bA_B^{-1}\begin{pmatrix}
1 & \phantom{-}1 &   & \\
1 & -1 &   &  \\
  &  & 1 & \phantom{-}1 \\
  &  & 1 & -1
\end{pmatrix}.
\]
of $\bT$. Let $\bT'$ be the matrix obtained from $\bT$ by removing either column $\br$ or column $\bs$. One can compute that $\bT'$ is bimodular but not unimodular, however both $\bT'/\bp$ and $\bT'/\bq$ are unimodular. Hence $\bT'/(x\bp + (x+1)\bq)$ is $\polyS$-modular. All elements of $\polyS$ are attained (up to sign) as maximal subdeterminants of this matrix, as shown in the following table:

\begin{center}
\renewcommand{\arraystretch}{1.2}
\begin{tabular}{p{4cm}|p{1cm} p{1cm} p{1cm} p{1cm} p{1cm}}
\raggedleft $\left|s\right|$: & \centering $0$ & \centering $1$ & \centering $x$ & \centering $x+1$ & \centering $2x+1$\tabularnewline
\hline 
\raggedleft number of occurences: & \centering 76 & \centering 36 & \centering 48 & \centering 48 & \centering 12\tabularnewline
\end{tabular}
\end{center}

\section{Open Directions}
There are several interesting directions that would allow us to generalize the range of applications of the framework introduced in this paper. 
A first step relates to Section~\ref{sec_smallest_nontrivial_cases}. We believe that one could establish a theory of what sets $\polyS\subseteq \polyring$ of polynomials are interesting because of their combinatorial or algebraic properties. For instance, for our most general set $\polyS = \pm \lbrace 0,1,\indet,\indet + 1, 2\indet + 1\rbrace$ of allowed determinants, it seems reasonable to add a determinant that is currently attained by a forbidden submatrix, i.e., an element in $F(\polyS)$.

So far all recognition algorithms in Section~\ref{sec_recognition} are based on characterizations that relate totally $\polyS$-modular matrices to totally unimodular matrices. In light of the celebrated Robertson-Seymour theorem \cite{ROBERTSON2004325}, a first step to establish results beyond Seymour's decomposition could be to classify for which sets $\polyS$ the number of forbidden submatrices is finite. Initial computations show that this is true for totally $\pm\lbrace 0,\indet,\indet + 1\rbrace$-modular matrices \cite[Section~1.1]{celayakuhlmannweismantel2024matricespolynomial} or when $0\notin\polyS$ \cite{ARTMANN2016635}. However, this is not true for our general set $\polyS=\pm\lbrace 0,1,\indet,\indet+1,2\indet+1\rbrace$. Even with entries restricted to $\lbrace \indet,\indet + 1\rbrace$ we can find an infinite sequence of forbidden submatrices; see the first row of Figure \ref{fig_forbidden minors}. Interestingly, if one removes $\pm1$ from $\polyS$ and passes to $\pm\lbrace 0,\indet,\indet+1,2\indet+1\rbrace$, then we are only aware of finitely many forbidden minors; see the second row of Figure \ref{fig_forbidden minors}. It is open whether this list is complete. A positive answer might support the following conjecture for a finite set $\polyS\subseteq\polyring$ such that $s\in S$ implies $-s\in S$: there exists a finite list of forbidden minors if $\pm1\notin \polyS$. 

\begin{figure}
	\centering
	\includegraphics[scale=.75]{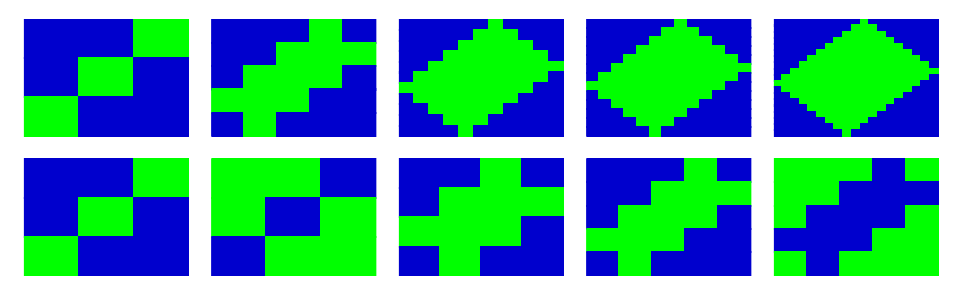}
	\caption{The blue boxes depict the value $\indet$ and the green boxes $\indet+1$ or vice-versa. The first row of matrices corresponds to the first five elements of an infinite sequence of matrices that can be obtained by generalizing the diamond pattern with dimension congruent to $3$ and $5$ modulo $8$. It can be shown that those infinitely many matrices are forbidden submatrices for $\pm\lbrace 0,1,\indet,\indet+1,2\indet+1\rbrace$.The matrices in the last row correspond to the only five forbidden submatrices for $\pm \lbrace 0,\indet,\indet+1,2\indet+1\rbrace$ that are known to us when $n\geq 3$.}
	\label{fig_forbidden minors} 
\end{figure}

The optimization result presented in this paper rely on Lemmas~\ref{lemma_bounded_subdeterminants} and \ref{lemma_reduction}, which reduce optimization over a totally $\pm\lbrace 0,1,\indet,\indet + 1,2\indet + 1\rbrace$-modular matrix to an optimization problem over a totally unimodular system with an additional constraint such that the resulting matrix is bimodular. A major hurdle in studying optimization problems for more general $\polyS$ is that currently there exists no polynomial time algorithm that solves integer linear programs with bounded subdeterminants beyond two. In view of Section~\ref{sec_bimodular}, where we show that $\pm\lbrace 0,1,\indet,\indet + 1,2\indet + 1\rbrace$-modular matrices arise as special bimodular matrices, it would be natural to consider analogous sets of polynomials that correspond to special families of $3$-modular matrices and allow for a decomposition into a totally unimodular matrix plus an additional row.

	\bibliography{references}

\end{document}